\documentclass[a4paper]{amsart}

\usepackage{amsmath,amssymb}
\usepackage{enumerate}
\usepackage{xspace}
\usepackage{stmaryrd}
\usepackage{listings}
\usepackage{color}
\usepackage{graphicx}
\usepackage{comment}

\numberwithin{equation}{section}

\newcommand{\Cetad}{C_{\underline{\eta}}}
\newcommand{\Cetau}{C_{\overline{\eta}}}
\newcommand{\Cfu}{C_{\overline{{\boldsymbol g}}}}
\newcommand{\fes}{\mathbb{V}}
\newcommand{\norm}[1]{\ensuremath{\left|#1\right|}}
\newcommand{\Norm}[1]{\ensuremath{\left\|#1\right\|}}
\newcommand{\hu}{{\bf {\hat u}}}
\newcommand{\eg}{e.g., }
\newcommand{\ie}{i.e., }
\newcommand{\leb}[1]{\ensuremath{L_{#1}}}
\newcommand{\jump}[1]{\ensuremath{\left\llbracket #1\right\rrbracket}}
\renewcommand{\d}{\operatorname{d}}
\newcommand{\sob}[2]{\ensuremath{W^{#1}_{#2}}}

\newcommand{\pdx}{\partial_x}
\newcommand{\bu}{{\bf u}}
\newcommand{\bv}{{\bf v}}
\newcommand{\bfl}{{\bf f}}
\newcommand{\bh}{{\bf h}}
\newcommand{\tu}{\tilde {\bu}}
\newcommand{\bw}{{\bf w}}
\newcommand{\bR}{{\bf R}}
\newcommand{\bg}{{\bf g}}
\newcommand{\bG}{{\bf G}}
\newcommand{\ter}{\hu^t}
\newcommand{\str}{\hu^{st}}

\newcommand{\cO}{\mathcal{O}}
\newcommand{\cP}{\mathcal{P}}
\newtheorem{lemma}{Lemma}[section]
\newtheorem{theorem}{Theorem}[section]

\newtheorem{assumption}[lemma]{Assumption}
\newtheorem{corollary}[lemma]{Corollary}
\theoremstyle{definition}
\newtheorem{definition}[lemma]{Definition}
\theoremstyle{definition}
\newtheorem{remark}[lemma]{Remark}

\title[]{A posteriori analysis of fully discrete method of lines DG schemes for systems of conservation laws}

\author{Andreas Dedner}
\address[Andreas Dedner]{\newline 
Mathematics Institute\newline
Zeeman Building\newline
University of Warwick\newline
Coventry CV4 7AL\newline
UK}
\curraddr{}
\email{a.s.dedner@warwick.ac.uk}

 \author{Jan Giesselmann} 
 \address[Jan Giesselmann]{\newline
  Institute of Applied Analysis and Numerical Simulation\newline
 University of Stuttgart\newline
 Pfaffenwaldring 57\newline
 D-70563 Stuttgart\newline
   Germany} 
 \curraddr{}
 \email{{jan.giesselmann@mathematik.uni-stuttgart.de}} 
 \thanks{
JG partially supported by the 
 German Research Foundation (DFG) via SFB TRR 75 `Tropfendynamische Prozesse unter extremen Umgebungsbedingungen'}

\date{\today}
\begin{document}

\begin{abstract}
  We present reliable a posteriori estimators for some
  fully discrete schemes applied to
  nonlinear systems of hyperbolic conservation laws in one space dimension with strictly convex entropy. 
  The schemes are based on a method of lines approach combining discontinuous Galerkin spatial discretization
  with single- or multi-step methods in time. The construction of the
  estimators requires a reconstruction in time for which we present a very
  general framework first for odes and then apply the approach to
  conservation laws. The reconstruction does not depend on the actual method used for evolving
  the solution in time. Most importantly it covers in addition to implicit methods
  also the wide range of explicit methods typically used to solve conservation laws. 
  For the spatial discretization, we allow for standard choices of
  numerical fluxes. We use \emph{reconstructions} of the discrete solution together
  with the \emph{relative entropy} stability framework,
  which leads to error control in the case of smooth solutions. We study
  under which conditions on the numerical flux the estimate is of optimal
  order pre-shock. 
  While the estimator we derive is
  computable and valid post-shock for fixed meshsize, it will blow
  up as the meshsize tends to zero.  This is due to a breakdown of the 
  relative entropy framework when discontinuities develop.
  We conclude with some numerical benchmarking 
  to test the robustness of the derived estimator.
\end{abstract}

\maketitle

\section{Introduction}
Systems of hyperbolic balance laws
are widely used in continuum mechanical modelling of processes in which 
higher order effects like diffusion and dispersion can be neglected, with the Euler- and the shallow water equations being prominent examples.
A particular feature of these equations is the breakdown of smooth solutions to initial value problems
for generic (smooth) initial data after finite time.
After this 'shock formation' discontinuous weak solutions are considered and attention is 
restricted to those satisfying a so-called entropy inequality.
Due to the interest in discontinuous solutions finite volume and discontinuous Galerkin (DG) spatial 
discretizations in space are state of the art \cite{GR96,Kro97,LeV02,HW08}.
In case of nonlinear systems the amount of theoretical results backing up these schemes
is quite limited. 
 A posteriori results for systems were derived in \cite{Laf04,Laf08} for  front tracking and Glimm's schemes, see also \cite{KLY10}.
 A priori estimates for fully discrete Runge-Kutta DG schemes were obtained in \cite{ZS06}
and an a posteriori error estimator for semi-discrete schemes was introduced in \cite{GMP_15}.
Other a priori and a posteriori results using the relative entropy method include 
\cite{AMT04,JR05,JR06}.
In \cite{HH02} the authors derive a posteriori estimates for space-time DG schemes in a goal oriented framework,
provided certain dual problems are well-posed.
All these results deal mainly with the pre-shock case.
Arguably, this is due to the following reason: the well-posedness of generic initial value problems
for (multi-dimensional) systems of hyperbolic conservation laws is quite open.
While uniqueness of entropy solutions was expected by many researchers for a long time, it was shown 
in recent years \cite{Chi14,DS10} that entropy solutions of nonlinear multi-dimensional systems are not unique, in general.
This severely restricts the range of cases in which a priori error estimates or convergent error estimators can be expected.
This is in contrast to the situation for scalar hyperbolic problems for which 
a priori convergence rates \cite{EGGH98} and convergent a posteriori error estimators
\cite{GM00,KO00,DMO07} are available post-shock; based on a more discriminating notion of entropy solution
\cite{Kru70}.
Still, the simulation of multi-dimensional hyperbolic balance laws is an important field,
and the numerically obtained results coincide well with experimental data.
In the simulations error indicators based on either entropy dissipation of the numerical
solution \cite{PS11} or nodal super convergence \cite{AABCP11} are used.
Our goal is to complement these results by a rigorous error estimator which, for certain types of numerical fluxes, is of optimal order, i.e.,
of the order of the true error, as long as the solution is
smooth. In case there is no Lipschitz continuous solution, but (possibly several) discontinuous entropy solutions, our estimator is still an upper bound for the error of the method, but it does 
not converge to zero under mesh refinement.
Our work is based on the results for semi-(spatially)-discrete methods in \cite{GMP_15} which we extend in two directions:
Firstly the results obtained here account for fully discrete, e.g. Runge-Kutta discontinuous Galerkin type, schemes and secondly we treat a larger class of numerical fluxes than was treated in \cite{GMP_15}.
Our results are optimal for a large class of central fluxes (e.g. of Richtmyer or Lax-Wendroff type)
augmented with stabilization approaches like artificial viscosity or flux
limiting \cite{Lapidus:1967:ArtVisc,Toro:2000:FLIC}. 
Furthermore, we also prove optimal convergence for Roe type numerical
fluxes for systems of conservation laws.

Our work is based on reconstructions in space and time, see  \cite{Mak07} for a general
exposition on the idea of reconstruction based error estimators, and the relative entropy
stability theory, going back to \cite{Daf79,Dip79}.
While the idea of reconstruction based a posteriori error estimates was extensively employed for implicit methods, see \cite[e.g.]{MN06}, it has not been used for explicit methods before.
Thus, we will describe our temporal reconstruction approach for general systems of odes first.
Our approach for reconstruction in the context of odes (which might be semi-discretizations of PDEs)
differs from other error estimation approaches \cite[e.g.]{Hig91,MN06} by being based on Hermite interpolation and by 
using information from old time steps. An advantage of this approach is
that the reconstruction does not depend on the time stepping method used
(covering for example both general IMEX type Runge-Kutta or multi-step
methods). The spatial reconstruction follows the approach first described in
\cite{GMP_15}. As mentioned above we extend the class of methods for which
optimal convergence of the estimate can be shown. Furthermore, we prove that
within the class of reconstructions considered here,
our restriction on the flux is not only sufficient for optimal convergence
but also necessary. We also show for a Roe type flux that whether the error estimator is suboptimal or  optimal
depends on the choice of reconstruction.


The layout of the rest of this work is as follows. In Section \ref{sec:ode}
we introduce a reconstruction approach for general  single- or multi-step discretizations of odes
and show that it leads to residuals of optimal order, \ie of the same order as 
the error.
This approach is applied to fully discrete discontinuous Galerkin schemes approximating systems 
of hyperbolic conservation laws in one space dimension endowed with a strictly convex entropy in Section \ref{sec:claws}.
We present numerical experiments in Section \ref{sec:num}.


\section{Reconstructions and error estimates for odes}\label{sec:ode}
In this section we introduce reconstructions for general  single- or multi-step methods approximating initial value problems for first order systems of odes.
The general set-up will be an initial value problem
\begin{equation}\label{ode} \d_t \bu = \bfl(t,\bu), \ \text{ on } (0,T),\ \bu(0)=\bu_0 \in \mathbb{R}^m\end{equation}
for some finite time $T>0$ and $\bfl: [0,T]\times \mathbb{R}^m\longrightarrow \mathbb{R}^m.$
We will usually assume that $\bfl$ is at least Lipschitz but we will specify our
regularity assumptions on $\bfl$ later.

Our reconstruction approach is based on Hermite (polynomial) interpolation
and the order of the employed polynomials depends on the convergence order of the method.

Our aim is to construct a continuous reconstruction of the solution to the
ode which does not depend on the actual method used to obtain the values $\{\bu^n\}_{n=0}^N$
approximating the exact solution at times $\{ t_n\}_{n=0}^N$ with $0=t_0< t_1 < \dots < t_N=T.$
Also note that, while $\{\bu^n\}_{n=0}^N$ are good approximations of
the exact solution at the corresponding points in time, this is usually not true for intermediate values 
$\{\bu^{n,i}\}_{n=0}^{N-1}$ which are frequently computed during, \eg  Runge-Kutta time-steps.
It is therefore in general a non trivial task to include these in the reconstruction.

\subsection{Reconstruction}\label{grks}
To define our reconstruction we need to introduce some notation:
By $\mathbb{P}_q(I,V)$ we denote the set of polynomials of degree at most $q$ on some interval $I$ with values in some vector space $V$ and
$\fes_q$ denotes a space of (possibly discontinuous) functions which are piece-wise polynomials of degree less or equal $q,$ \ie
\begin{equation}
 \fes_q := \{ {\bf w}: [0,T] \rightarrow \mathbb{R}^m \, : \, {\bf w}|_{(t_{n-1},t_{n})} \in \mathbb{P}_q((t_{n-1},t_{n}),\mathbb{R}^m) \text{ for } n=1,\dots, N\}.
\end{equation}
For any $n$ and ${\bf w} \in \fes_q$ we define traces by
\begin{equation}\label{def:traces}
 {\bf w}(t_n^\pm):= \lim_{s \searrow 0} {\bf w}(t_n \pm s)
\end{equation}
and element-wise derivatives $\d^e_t {\bf w} \in \fes_{q-1}$ by
\begin{equation} 
(\d^e_t {\bf w})|_{(t_n,t_{n+1})} = \d_t ( {\bf w}|_{(t_n,t_{n+1})}).
\end{equation}

Let $\{\bu^n\}_{n=0}^N$ denote the approximations of the solution of \eqref{ode} at times $\{t_n\}_{n=0}^N$ computed using some single- or multi-step method. 
We will define $\hu$ as a $C^0$ or even $C^1$-function which is piece-wise polynomial and whose polynomial degree matches the convergence order of the method.
To define $\hu|_{[t_n,t_{n+1}]}$ the information $\bu^n, \bu^{n+1}, \bfl(\bu^n), \bfl(\bu^{n+1})$ is readily available, but 
for polynomial degrees above $3$ we need additional conditions. 
Let $\hu^n$ denote the polynomial which coincides with $\hu$ on $[t_n,t_{n+1}].$ Note that $\hu^n$ differs from $\hu|_{[t_n,t_{n+1}]}$ by being defined on all of $\mathbb{R}.$
Then, $\hu|_{[t_n,t_{n+1}]}$ can be obtained by prescribing values of $\hu^n$, $(\hu^n)'$ at additional points or by prescribing additional derivatives of $\hu^n$ at $t_n$ and $t_{n+1}$ or by a combination of both approaches.

For $(p,d,r) \in \mathbb{N}_0^3$ we denote by $H(p,d,r)$  the reconstruction which fixes the value and the first $d+1$ derivatives of $\hu^n $ at $t_{n-p},\dots, t_n$ and the value and the first $r+1$ derivatives of $\hu^n$ at $t_{n+1}.$

Before we make this more precise let us note that we can express higher order derivatives of the solution $\bu$ to \eqref{ode} by evaluating $\bfl$ and its derivatives, \eg
\begin{equation}\label{fn2}
 \d_t^2 \bu(t_n)= \partial_t \bfl(t_n,\bu(t_n)) + \operatorname{D} \bfl(t_n,\bu(t_n)) \bfl(t_n,\bu(t_n)),
\end{equation}
where $\operatorname{D} \bfl$ is the Jacobian of $\bfl$ with respect to $\bu,$
provided $\bfl$ is sufficiently regular.
We denote the corresponding expression for
$\d_t^k \bu(t_n)$ by $\bfl_k^n(\bu(t_n)).$
Note that instead of $\bu(t_n)$ we may also insert $\bu^n$ into this expression.

\begin{remark}[Regularity of $\bfl$]
Subsequently we impose $\bfl \in C^{\max{\{d,r\}}}((0,T)\times \mathbb{R}^m, \mathbb{R}^m)$. 
This is the amount of regularity required to define our reconstruction.
It is not sufficient for the Runge-Kutta method to have (provably) an
error of order $\cO(\tau^q)$, where $\tau$ is the maximal time step size.
When investigating the optimality of the residual in Section \ref{subs:oode}
we will indeed require more regularity of $\bfl.$
\end{remark}

Now we are in position to define our reconstruction:

\begin{definition}[Reconstruction for odes]\label{def:grec}
The $H(p,d,r)$ reconstruction  $\hu \in \fes_q$ with $q=(d+2)(p+1)+r+1$ is determined  by
\begin{equation}\label{grec}
\begin{split}
 (\d_t)^k \hu^n(t_j) &=\bfl_k^j(\bu^j), \text{ for } k=0, \dots, d+1 \text{ and } j=n-p, \dots, n \\
 (\d_t)^k \hu^n(t_{n+1})   &=\bfl_k^{n+1}(\bu^{n+1})\text{ for } k=0, \dots, r+1.
 \end{split}
\end{equation}
\end{definition}

\begin{remark}[Start up]
 Note that (strictly speaking) the $H(p,d,r)$ reconstruction is not defined on $[t_0,t_p].$
 However, computing the numerical solution for the first $p+1$ time steps we may use conditions at $\{t_0,\dots,t_p\}$
 to define $\hu|_{[t_0,t_p]}$ in an analogous way.
\end{remark}

By standard results on Hermite interpolation we know that:

  \begin{lemma}[Properties of reconstruction]
   Any $H(p,d,r)$ reconstruction $\hu \in \fes_q$ with $q= (d+2)(p+1)+r+1$
   as given in Definition \ref{def:grec}   is well-defined, computable and
   $\sob{1}{\infty}$ in time.
   For higher values of $d,r$ we even have that $\hu$ is $\min\{d+1,r+1\}$-times continuously differentiable.
  \end{lemma}

\begin{remark}[Particular methods]
 In the work at hand we will restrict our attention to two classes of methods: Either methods of type $H(p,0,0)$ or methods of types $H(0,d,d)$ and $H(0,d,d-1).$
 Both methods need the same amount of extra storage to arrive at the same order.
 Methods of type $H(p,0,0)$ have the advantage that no derivatives of $\bfl$ need to be evaluated or to be approximated. 
 In case, derivatives of $\bfl$ are explicitly known and can be cheaply evaluated $H(0,d,d)$ has the advantage that no information from old time steps has to be accessed.
 We choose the numbers of conditions imposed at $t_n$ and $t_{n+1}$ to be equal, so that all the information created at $t_{n+1}$ can be reused when $\hu|_{[t_{n+1},t_{n+2}]}$ is computed.
\end{remark}

 \begin{remark}[Derivatives of $\bfl$]
 In many cases of importance, \eg $\bfl$ being a spatial semi-discretization
 of a system of hyperbolic conservation laws,
 the explicit computation of derivatives of $\bfl$  might be infeasible or numerically expensive.
 Thus, for $H(0,d,d)$ with $d \geq 1$ it seems interesting to replace, \eg $\operatorname{D}\bfl(\bu)\bfl(\bu)$
 by an approximation thereof.
 We will elaborate upon this in Section \ref{subs:approxf}.\end{remark}

Using any of the methods $H(p,d,r)$ the reconstruction $\hu$ is explicitly (and locally) computable, continuous and piece-wise polynomial, thus
\begin{equation}\label{tres} {\bf R}:= \d_t \hu - \bfl(\hu) \in \leb{\infty}(0,T)\end{equation}
is computable.
Therefore, standard stability theory for odes implies our first a posteriori results.

\begin{lemma}[A posteriori estimates for odes]\label{lem:apode}
Let \eqref{ode} have an exact solution $\bu$ and let $\bfl$ be Lipschitz with respect to $\bu$, with Lipschitz constant $L,$ on a neighbourhood of the set of values taken
by $\bu$ and $\hu.$ Then,
\begin{equation}
 \Norm{ \bu - \hu }_{\leb{\infty}(0,T)} \leq (\norm{ \bu_0 - \hu(0)} + \Norm{{\bf R}}_{\leb{1}(0,T)}) e^{LT}
\end{equation}
and 
\begin{equation}
 \Norm{ \bu - \hu }_{\leb{2}(0,T)}^2 \leq (\norm{ \bu_0 - \hu(0)}^2 + \Norm{{\bf R}}^2_{\leb{2}(0,T)}) e^{(L+1)T}
\end{equation}
with ${\bf R}$ being defined in \eqref{tres}.
\end{lemma}

\begin{remark}[Influence of the Lipschitz constant of $\bfl$]
The appearance of the Lipschitz constant of $\bfl$ in the error estimators in Lemma 
\ref{lem:apode} cannot be avoided for general right hand sides $\bfl.$
However, for many particular right hand sides it can be avoided,
because better stability results for the underlying ode are available. 
A particular example are right hand sides deriving from the spatial discretization of systems of 
hyperbolic conservation laws which we will investigate in Section \ref{sec:claws}.
\end{remark}

\subsection{Optimality of the residual}\label{subs:oode}
In this section we investigate the order of ${\bf R}$. 
To this end we restrict ourselves to an equidistant time step $\tau>0.$
We also assume that there exists a constant $L>1$ such that for all $k,l \in \mathbb{N}$
\begin{equation}\label{Lip}
 \norm{\partial_t^l\operatorname{D}^k \bfl (t,{\bf y})  - \partial_t^l\operatorname{D}^k \bfl (t,\bar {\bf y})} 
 \leq L^{k+1} \norm{{\bf y} - \bar {\bf y}} \quad \forall \ t \in [0,T], \ {\bf y},\bar {\bf y} \in \mathbb{R}^m,
\end{equation}
where on the left hand side $\norm{\cdot}$ denotes an appropriate norm on $\mathbb{R}^{m^{k+1}}.$
\begin{remark}[Parameter dependence]
 Note that we explicitly keep track of the dependence of the subsequent estimates on $L,$ from \eqref{Lip},  while we 
 suppress all other constants using the Landau $\cO$ notation.
 This is due to the fact that we will apply the results obtained here to spatial semi-discretizations of hyperbolic conservation laws in Section
 \ref{sec:claws}. In that case the Lipschitz constant $L$  depends on the spatial mesh width,
 which for practical computations is of the same order of magnitude as the time step size.
\end{remark}

For a time integration method of order $q$ we consider a reconstruction
 of type $H(0,d,r)$ such that  $q=d+r+3$ and $r\in \{ d, d-1\}.$ In this way
the polynomial degree of the reconstruction coincides with the order of the method.
The analysis for residuals stemming from reconstructions of type $H(p,0,0)$ is analogous.

For estimating the residual we make use of two auxiliary functions:
Firstly, for $n=0,\dots,N-1$ we denote by $\tu^n$ the exact solution to the initial value problem
\begin{equation}\label{auxivp}
\d_t \tu^n(t) = \bfl(t,\tu^n(t)) \quad \text{ on } (t_n,t_{n+1}), \quad \tu^n(t_n)=\bu^n.
\end{equation}
We assume from now on that $\bfl$ is indeed regular enough for the method at hand
to be convergent of order $q$. In particular, we assume consistency errors to be of order $q+1$, \ie
\[ \norm{ \tu^n(t_{n+1}) - \bu^{n+1}} = \cO(\tau^{q+1}).\]

 Secondly, for $n=0,\dots,N-1$ let $\bh^n\in \mathbb{P}_q((t_n,t_{n+1}),\mathbb{R}^m)$ be the Hermite interpolation of $\tu^n$, \ie
\begin{equation}\label{herm}
\begin{split}
 (\d^e_t)^k \bh^n(t_n^+) &=\bfl_k^n(\tu^n(t_n)), \text{ for } k=0, \dots, d+1 \\
 (\d^e_t)^k \bh^n(t_{n+1}^-)   &=\bfl_k^{n+1}(\tu^n(t_{n+1}))\text{ for } k=0, \dots, r+1.
 \end{split}
\end{equation}
By standard results on Hermite interpolation we have
\begin{equation}
\label{hermest}
 \Norm{ \tu^n - \bh^n}_{\leb{\infty}(t_n,t_{n+1})} = \cO(\tau^{q+1})\quad \text{and} \quad
 \Norm{ \tu^n - \bh^n}_{\sob{1}{\infty}(t_n,t_{n+1})} = \cO(\tau^{q})
\end{equation}
for $n=0,\dots,N-1.$

Let us write $\bR^n$ instead of $\bR|_{(t_n,t_{n+1})}$ for brevity.
Because of \eqref{auxivp} we may rewrite \eqref{tres} as
\begin{multline}
 \label{tres2}
 \bR^n= \d_t(\hu|_{(t_n,t_{n+1})} - \bh^n) + \d_t (\bh^n - \tu^n)\\ - (\bfl(\hu|_{(t_n,t_{n+1})}) - \bfl(\bh^n)) - (\bfl(\bh^n) - \bfl(\tilde \bu^n))=: \bR_1^n + \bR_2^n + \bR_3^n +\bR_4^n
\end{multline}
and \eqref{hermest} immediately implies
\begin{equation}\label{Rest1}
 \Norm{ \bR_2^n}_{\leb{\infty}(0,T)} = \cO(\tau^{q}), \quad  
 \Norm{ \bR_4^n}_{\leb{\infty}(0,T)} = L \cO(\tau^{q+1}).
\end{equation}

Regarding, the estimates of $\bR_1^n,\bR_3^n$ we make use of the fact that
$\hu|_{(t_n,t_{n+1})} - \bh^n \in \mathbb{P}_q((t_n,t_{n+1}),\mathbb{R}^m)$ satisfies
\begin{equation}\label{difrec}
\begin{split}
 (\d^e_t)^k (\hu|_{(t_n,t_{n+1})}-\bh^n)(t_n^+) &=0  , \text{ for } k=0, \dots, d+1 \\
 (\d^e_t)^k (\hu|_{(t_n,t_{n+1})}-\bh^n)(t_{n+1}^-)   &=\bfl_k^{n+1}(\bu^{n+1}) -\bfl_k^{n+1}(\tu^n(t_{n+1})) \text{ for } k=0, \dots, r+1
 \end{split}
\end{equation}
and due to \eqref{Lip} and consistency of the underlying method
\begin{equation}\label{difrec1}
\bfl_k^{n+1}(\bu^{n+1}) -\bfl_k^{n+1}(\tu^n(t_{n+1})) = L^{k} \cO(\tau^{q+1})\text{ for } k=0, \dots, r+1.
\end{equation}

\begin{lemma}[Stability of Hermite interpolation]\label{lem:Hermite}
 Let $\bv \in \fes_q$ satisfy 
 \begin{equation}\label{hermlem}
\begin{split}
 (\d^e_t)^k \bv(t_n^+) &=  a_k , \text{ for } k=0, \dots, d+1 \\
 (\d^e_t)^k \bv(t_{n+1}^-)   &=b_k, \text{ for } k=0, \dots, r+1
 \end{split}
\end{equation}
with sequences $\{a_k\}_{k=0}^{d+1}, \{ b_k\}_{k=0}^{r+1}$ satisfying
\[ a_k, b_k \leq L^{k} \cO(\tau^{q+1})\]
for some $L,\tau>0.$
Then, it holds
\begin{equation}\label{eq:Hermite1}
 \Norm{ \bv}_{\leb{\infty} (0,T)} \leq \sum_{k=0}^q L^k \cO(\tau^{q+1+k})
\end{equation}
and 
\begin{equation}\label{eq:Hermite2}
 \Norm{\d_t \bv}_{\leb{\infty} (0,T)} \leq \sum_{k=0}^q L^k \cO(\tau^{q+k}).
\end{equation}
\end{lemma}

\begin{remark}[Approximation order]\label{rem:apporder}
 Note that the order with respect to $\tau$ of the right hand side in \eqref{eq:Hermite1}  is the same in case $L$ is a constant independent of $\tau$ as in case 
 $L$ satisfies an estimate of the form $L \leq C \tau^{-1}$ with some $C>0$ independent of $\tau$.
 The same is true for the right hand side of \eqref{eq:Hermite2}.
 
 This shows that we may replace $\bfl^n_k(\bu^n), \bfl^{n+1}_k(\bu^{n+1})$ in \eqref{grec} by (sufficiently accurate) approximations
 $\tilde \bfl^n_k[\bu^n],\tilde \bfl^{n+1}_k[\bu^{n+1}]$ thereof without compromising the quality of the reconstruction.
 In particular, we need for the residual to be of order $q$ that
 \[ \norm{\bfl^m_k (\bu^m) - \tilde \bfl^m_k [\bu^m] } = \cO(\tau^{q+1-k}).\] 
We write  $\tilde \bfl^m_k [\bu^m]$ instead of $\tilde \bfl^m_k (\bu^m)$ to indicate that this quantity might not only depend on $\bu^m,$
but also on $\bu^{m-1}, \bu^{m-2}$ etc.
 We will exploit this observation in Section \ref{subs:approxf}.
\end{remark}

\begin{proof}[Proof of Lemma \ref{lem:Hermite}]
 The proof considers $\bv$ on each interval separately.
The functions $\{\psi_i\}_{i=0}^q$ with
\[ 
 \begin{split}
  \psi_i(t)&= (t-t_n)^i \quad \text{for } i=0,\dots,d+1\\
  \psi_i(t)&= (t-t_n)^{d+1} (t- t_{n+1})^{i - d-1} \quad \text{for } i=d+2,\dots,q
 \end{split}
 \]
form a basis of $\mathbb{P}_q((t_n,t_{n+1}),\mathbb{R})$ and obviously
\begin{equation}\label{psiiest}
 \Norm{ \psi_i}_{\leb{\infty} (t_n,t_{n+1})} = \cO(\tau^i); \quad
 \Norm{\d_t \psi_i}_{\leb{\infty} (t_n,t_{n+1})} =\cO(\tau^{i-1}).
\end{equation}

In this basis the coefficients of $\bv$ are given by divided differences, \ie 
\begin{equation}\label{lb}
 \bv|_{(t_n,t_{n+1})} = \sum_{i=0}^{d+1} \bar \bv [\underbrace{t_n, \dots, t_n}_{(i+1)-\text{times}}] \psi_i + \sum_{i=d+2}^{q} \bar \bv [\underbrace{t_n, \dots, t_n}_{(d+2)-\text{times}},\underbrace{t_{n+1}, \dots, t_{n+1}}_{(i-d-1)-\text{times}}] \psi_i
\end{equation}
with
\begin{equation}\label{dd1}
\begin{split}
 \bar \bv [\underbrace{t_n, \dots, t_n}_{i-\text{times}}] &= a_{i-1} \quad \text{for } i=1,\dots, d+2\\
 \bar \bv [\underbrace{t_{n+1}, \dots, t_{n+1}}_{j-\text{times}}] &= b_{j-1} \quad \text{for } j=1,\dots, r+2
  \end{split}
\end{equation}
and 
\begin{equation}\label{dd2}
  \bar \bv [\overbrace{t_n, \dots, t_n}^{i-\text{times}},\overbrace{t_{n+1}, \dots, t_{n+1}}^{j-\text{times}}] =
  \frac{ \bar \bv [\overbrace{t_n, \dots, t_n,}^{(i-1)-\text{times}}\overbrace{t_{n+1}, \dots, t_{n+1}}^{j-\text{times}}] - 
  \bar \bv [\overbrace{t_n, \dots, t_n,}^{i-\text{times}}\overbrace{t_{n+1}, \dots, t_{n+1}}^{(j-1)-\text{times}}]}{\tau}
  \end{equation}
  for  $i=1,\dots, d+2,\ j=1,\dots, r+2 .$
  
 In particular, \eqref{lb} shows that the coefficient of $\psi_i$ is a divided difference with $i+1$ arguments.
 Our assumptions on $a_k, b_k$ imply that divided differences containing only one argument $j$-times are of order
 $L^{j-1} \cO(\tau^{q+1}),$ in particular
 \begin{equation}\label{hermest1}
  \bar \bv [\underbrace{t_n, \dots, t_n}_{(j+1)-\text{times}}] \leq L^{j} \cO(\tau^{q+1}),
 \end{equation}
for $j=0, \dots, d+1$.
Moreover,  divided differences containing $j_1$-times $t_n$ and $j_2$-times $t_{n+1}$ are bounded by terms of the form
 \[ \sum_{l=1}^{\max\{j_1,j_2\}} L^{l-1} \cO(\tau^{q+1-j_1-j_2+l})\]
 due to our assumptions on $a_k, b_k;$ \eqref{dd1} and \eqref{dd2}.
Thus, shifting the summation index,
\begin{equation}\label{hermest2}
\bar \bv [\underbrace{t_n, \dots, t_n}_{(d+2)-\text{times}},\underbrace{t_{n+1}, \dots, t_{n+1}}_{(i-d-1)-\text{times}}] \leq \sum_{l=0}^{d+2} L^{l} \cO(\tau^{q-i+l+1}).
\end{equation}
We obtain the assertion of the Lemma by combining
\eqref{psiiest}, \eqref{hermest1} and \eqref{hermest2}.
\end{proof}

Combining \eqref{difrec}, \eqref{difrec1} and Lemma \ref{lem:Hermite} we obtain
\begin{corollary}[Bounds on residuals]\label{cor:r13}
 Let $\bR_1^n, \bR_3^n$ be defined as in \eqref{tres2}, then
 \begin{equation}
 \Norm{\bR_1^n}_{\leb{\infty} (t_n,t_{n+1})} \leq \sum_{k=0}^q L^k \cO(\tau^{q+k})
\end{equation}
and 
\begin{equation}
 \Norm{\bR_3^n}_{\leb{\infty} (t_n,t_{n+1})} \leq \sum_{k=0}^q L^{k+1} \cO(\tau^{q+1+k})
\end{equation}
for $n=0,\dots, N-1.$
\end{corollary}

Now we are in position to state the main result of this subsection:
\begin{theorem}[Optimality]\label{thrm:ode}
 Let $\hu\in \fes_q$ be the $H(p,d,r)$ reconstruction with $(p+1)(d+2)+r+1=q$ of the solution to a $q$-th order single- or multi-step method, approximating \eqref{ode}, defined in \eqref{grec}.
 Then, the residual $\bR$ defined in \eqref{tres} satisfies
 \begin{equation}
   \Norm{\bR}_{\leb{\infty} (0,T)} \leq \sum_{k=0}^{q+1} L^{k} \cO(\tau^{q+k}).
 \end{equation}

\end{theorem}

\begin{proof}
 The assertion of the Theorem follows upon combining \eqref{tres2}, \eqref{Rest1} and Corollary \ref{cor:r13}.
\end{proof}

\subsection{Approximation of $\bfl$ derivatives}\label{subs:approxf}
The expressions $\bfl^n_k$ used in \eqref{grec} might not be explicitly computable in many practical cases, \eg in case the right hand side stems from the spatial 
discretization of a system of hyperbolic conservation laws.
We have already observed in Remark \ref{rem:apporder} that we may replace $\bfl^n_k$ in \eqref{grec} by some approximation.
Obviously, for any $n$ there is only an interest in replacing $\bfl^n_k$ for $k \geq 2$ as $\bfl_0^n(\bu^n)=\bu^n$ and $\bfl_1^n(\bu^n)=\bfl(\bu^n)$ are readily computable and are
probably computed anyway by the method used for time integration.
Note that for $H(p,0,0)$ reconstructions only $\bfl_0^n, \bfl_1^n$ appear in \eqref{grec}.

As the required order of accuracy of the approximations and also the different values of $k$ for which $\bfl^n_k$ appears in \eqref{grec} depend on $q$ we will 
present different approximation approaches for different values of $q >3.$

\subsubsection{Directional derivatives}\label{subs:af:dd}
In time integration methods of order $q=4,5$ reconstructions of types $H(0,1,0)$ and $H(0,1,1)$ require
$\bfl^m_2(\bu^m)$ for $m =n,n+1$, where
\[\bfl^m_2(\bu^m) = \partial_t \bfl(t_m,\bu^m) + \operatorname{D} \bfl(t_m,\bu^m) \bfl(t_m,\bu^m),\]
see \eqref{fn2}
As we have seen in Remark \ref{rem:apporder} we may replace $\bfl^m_2 $ by $\tilde \bfl^m_2 $
as long as the error is of order $\cO(\tau^{q-1})$ such that $\cO(\tau^4)$ is admissible for $q=4,5.$
Such an error is achieved by
\begin{multline} \tilde \bfl_2^m[\bu^m] := \frac{ \bfl(t_m + \tau^2,\bu^m)-  \bfl(t_m - \tau^2,\bu^m)}{2\tau^2}\\
+\frac{ \bfl(t_m, \bu^m + \tau^2 \bfl(t_m,\bu^m))- \bfl(t_m, \bu^m - \tau^2 \bfl(t_m,\bu^m))}{2\tau^2} \end{multline}
for $m=n, n+1.$
Note that for computing $\tilde \bfl_2^{n+1}[\bu^{n+1}]$ four additional
$\bfl$ evaluations (two if $\bfl$ does not depend on $t$) are 
required as the time integration scheme computes $ \bfl(t_{n+1},\bu^{n+1}).$
The value $\tilde \bfl_2^{n+1}[\bu^{n+1}]$ computed for use at the right boundary in the time step from $t_n$ to $t_{n+1}$ can be reused in the next time step.
Thus, this reconstruction approach requires four (two) additional $\bfl$ evaluations per time step.

\subsubsection{Finite differences}\label{subs:af:fd}
Another way to obtain approximations of $\bfl_k^{n+1}(\bu^{n+1})$ is to use finite difference 
approximations of $d_t^{k-1} \bfl(t,\tilde \bu(t))|_{t_{n+1}}$ in which $\tilde \bu$ is some function interpolating $(t_{n-m}, \bu^{n-m}), \dots, (t_n, \bu^n), (t_{n+1},\bu^{n+1})$ for a sufficient number of points.
It is preferable to choose those interpolated points as points which do not lie in the future of the point $t_{n+1}$ as this avoids the need to compute $\bfl$ evaluations before they are needed by the time integration scheme.

To be precise, let us elaborate the approach for $q=4,5$. 
We use the following backward finite difference stencil for first order derivatives
\begin{equation}\label{def:bfd}
\tilde \bfl^{n+1}_2[\bu^{n+1}]:=
\frac{\frac{25}{12} \bfl_{n+1} - 4 \bfl_n + 3\bfl_{n-1} - \frac{4}{3}\bfl_{n-2} + \frac{1}{4}\bfl_{n-3}}{\tau} \quad \text{ for } n \geq 3,
\end{equation}
where we used the abbreviation
\begin{equation}
\bfl_m := \bfl(t_m, \bu^m).
\end{equation}

This approach does not require any additional $\bfl$ evaluations but creates a need 
to store $3$ previous $\bfl$ evaluations.
While this creates (nearly) no additional overhead in multi-step schemes, it is a certain overhead in one-step schemes like Runge-Kutta methods.
The formula \eqref{def:bfd} does not allow for the computation of $\tilde \bfl^{m}_2[\bu^{m}]$
for $m=0,\dots,3$
so that a computation of our reconstruction on the first four time steps is not possible.
However after performing the first six time steps we may compute  $\tilde \bfl_2^{0}[\bu^{0}], \dots, \tilde \bfl_2^{3}[\bu^{3}]$ using forward and central finite difference stencils.

Of course higher order finite-difference schemes can be used for the first
and higher order derivatives. For example using the $6th$ order one sided
differences schemes for the first derivative and the $5th$ order scheme for
the second derivative described in \cite{For88} allows for a reconstruction
with $q=7$. Using one sided differences makes it possible to reuse one 
approximation from one time step to the next. These methods require to
store $\bfl_{n-6},\dots,\bfl_{n-1}$.
Note that the procedures discussed in \cite{For88} also allow us to construct higher order finite difference stencils possibly including time steps of different lengths.
However, for each combination of lengths of intervals a new stencil needs to be derived.

\begin{remark}[Comparison of storage demands]\label{rem:stor}
 In case $H(0,d,d)$ methods use the finite difference stencils described above they do not need any additional $\bfl$ evaluations,
 but require the storage of $\bfl$ evaluations from previous time steps.
 In this sense they become comparable to $H(p,0,0)$ schemes, and we may compare the storage demands of both schemes.
In case of $q=5$ we may use $H(0,1,1)$ requiring the storage of $3$ previous $\bfl$ evaluations or $H(1,0,0)$ requiring the storage of $1$ previous $\bfl$ evaluation.
In case of $q=7$ we may use $H(0,2,2)$ requiring the storage of $6$ previous $\bfl$ evaluations or $H(2,0,0)$ requiring the storage of $2$ previous $\bfl$ evaluations.

From this perspective it seems that $H(p,0,0)$ schemes are more efficient than $H(0,d,d)$ schemes.
\end{remark}

\begin{remark}[Hermite-Birkhoff interpolation]
As a final remark we note that Hermite-Birkhoff interpolation could be used as
well. As an example the reconstruction suggested in \cite{BKM12}, is similar to the approach presented in the work at hand. 
That reconstruction corresponds to fixing $\hu \in \mathbb{V}_3$ by prescribing
three interpolation conditions for
\begin{equation}\label{eq:agrec}
 \hu^n (t_n)  ,\;
 \d^2_t \hu^n(t_{n+\frac{1}{2}}) ,\;
 \hu^n (t_{n+1})~.
\end{equation}
While the values at the end points are readily available some approximation
is needed for the value of the second derivative at the interval midpoint.
For this condition the following approximation is used
$$ \frac{\d}{\d t} \bfl(t_{n+\tfrac{1}{2}},\bu(t_{n+\frac{1}{2}})) \approx
    \frac{1}{2\tau}\big( \bfl(t_{n+1},\bu^{n+1}) - \bfl(t_n,\bu^{n}) \big) $$
resulting in an optimal scheme for $q=2$. 
A similar approximation order could be obtained by considering our reconstruction framework for the choice
$H(0,0,-1)$ where the $-1$ should indicate that we prescribe the value but not any derivative of $\hu$ at $t_{n+1}.$
Further Hermite-Birkhoff
reconstructions for special Runge-Kutta methods are derived in
\cite{EJNT86,Hig91}.
\end{remark}



\section{Estimates for fully discrete schemes for conservation laws}\label{sec:claws}

Let us consider  a spatially one dimensional, hyperbolic  system of $m\in \mathbb{N}$  conservation laws
on the flat one dimensional torus $\mathbb{T}$ complemented with initial
data $\bu_0 : \mathbb{T} \rightarrow \mathcal{U} \subset \mathbb{R}^m$
where the state space $\mathcal{U}$ is an open set:
\begin{equation}\label{claw}
 \partial_t \bu + \partial_x {\bf g}(\bu) =0 \text{ on } (0,T) \times \mathbb{T}, \quad \bu(0,\cdot) = \bu_0 \ \text{ on } \mathbb{T}.
\end{equation}
We assume the flux function ${\bf g}$ to be in $C^2(\mathcal{U}, \mathbb{R}^m).$

We restrict ourselves, to the case that \eqref{claw} is endowed with a strictly convex entropy, entropy flux pair, \ie    
 there exists a strictly convex $\eta \in C^1(\mathcal{U}, \mathbb{R})$ and $q \in C^1(\mathcal{U}, \mathbb{R})$ satisfying
\begin{equation}
 (\operatorname{D} \eta) \operatorname{D} {\bf g} = \operatorname{D} q.
\end{equation}

It is straightforward to check that any classical solution $\bu$ to \eqref{claw} satisfies the companion conservation law
\begin{equation}\label{claw2}
  \partial_t \eta(\bu) + \partial_x q(\bu) =0.
\end{equation}

Classically the study of weak solutions to \eqref{claw} is  restricted to so-called entropy solutions, see \cite[e.g.]{Daf10}.
\begin{definition}[Entropy solution]
 A weak solution $\bu \in \leb{\infty}((0,T) \times \mathbb{T},\mathcal{U})$ to \eqref{claw} is called an {\bf entropy solution}
 with respect to $(\eta,q),$ 
 if it weakly satisfies
 \begin{equation}\label{ei}
  \partial_t \eta(\bu) + \partial_x q(\bu) \leq0.
\end{equation}
\end{definition}

It was believed for a long time that entropy solutions are unique. 
While this is true for entropy solutions to scalar problems (satisfying a more discriminating entropy condition),
at least in multiple space dimensions entropy solutions to e.g. the Euler equations are not unique \cite{DS10,Chi14}.

While it is not entirely clear whether entropy solutions to (general) systems in one space dimension with generic initial data are unique 
it is well known that the entropy inequality \eqref{ei} gives rise to some stability results, which in particular imply
weak-strong-uniqueness, see \cite{Dip79,Daf79}.

\subsection{Reconstruction for fully discrete DG schemes}
In the sequel we study fully discrete schemes approximating \eqref{claw} employing a method of lines approach.
We assume that the spatial discretization is done using a DG method
with $q$-th order polynomials and that the temporal discretization is based
on some single- or multi-step method of order $r.$

We use decompositions  $-1= x_0 < x_1 < \dots < x_{M-1} < x_M=1$ of the spatial domain and $0= t_0 < t_1 < \dots < t_{N-1} < t_N=T$
of the temporal domain.
In order to account for the periodic boundary conditions we identify $x_0$ and $x_M.$
We define time steps $\tau_n:= t_{n+1}-t_n,$ a maximal time step $\tau:= \max_n \tau_n$,
spatial mesh sizes $h_{k+\frac{1}{2}}:= x_{k+1}-x_k,$ $h_k := (h_{k+\frac{1}{2}} + h_{k-\frac{1}{2}} )/2$ and a maximal and minimal spatial step
\[ h:= \max_k h_{k+\frac{1}{2}}, \quad h_{\min} := \min_k h_{k+\frac{1}{2}}\] 
where we assume that $\frac{h}{h_{\min}}$ is bounded for $h\to 0$.
We will write $\int_{\mathcal{T}}$ instead of $\sum_{i=1}^{M} \int_{x_{i-1}}^{x_{i}}.$

Let us introduce the piece-wise polynomial DG ansatz and test space:
\begin{equation}\label{def:dgs}
 \fes_q^s := \{ {\bf w} : [x_0,x_M]\rightarrow \mathbb{R}^m \, : \,{\bf w}|_{(x_{i-1},x_{i})} 
 \in \mathbb{P}_q((x_{i-1},x_{i}),\mathbb{R}^m) \text{ for }  1 \leq i \leq M\}
 \end{equation}
Then, the fully discrete scheme results as a single- or multi-step discretization of the semi-discrete scheme 
\begin{equation}\label{sde}
 \partial_t \bu_h = -{\boldsymbol f} (\bu_h)
\end{equation}
where the (nonlinear) map ${\boldsymbol f}: \fes_q^s \rightarrow \fes_q^s$ is defined by requiring that for all ${\bv}_h, {\boldsymbol \psi} \in \fes^s_q$ it holds
\begin{equation}\label{dgscheme}
 \int_\mathbb{T} {\boldsymbol f}({\bf v}_h) {\boldsymbol \psi}\d x = - \int_{\mathcal{T}}{\bf g}({\bf v}_h) \partial_x {\boldsymbol \psi} \d x 
 + \sum_{i=0}^{M-1} {\bf G}({\bf v}_h(x_i^-),{\bf v}_h(x_i^+)) \jump{{\boldsymbol \psi}}_i,
\end{equation}
 where $\bG:\mathcal{U} \times \mathcal{U}
 \rightarrow \mathbb{R}^m$ is a numerical flux function, $\jump{{\boldsymbol \psi}}_i := {\boldsymbol \psi}(x_i^-)-
 {\boldsymbol \psi}(x_i^+)$ are jumps and
  the notation for spatial traces is analogous to that for temporal traces, see \eqref{def:traces}.
We will specify our assumptions on the numerical flux in Assumption \ref{ass1}.

Suppose the fully discrete numerical scheme allows us to compute a sequence of approximate solutions at  points $\{t_n\}_{n=0}^N$ in time: $\bu_h^0, \bu_h^1, \bu_h^2,\dots, \bu_h^N \in \fes^s_q.$
In order to make sense of its reconstruction, we define for any vector space $V$ a space of piece-wise polynomials in time by
\begin{equation}\label{tspace}
  \fes_{r}^{t}(0,T; V) := \{ {\bf w} : [0,T] \rightarrow V \, :
  \,{\bf w}|_{ (t_n,t_{n+1})} \in \mathbb{P}_r((t_n,t_{n+1}),V)\}.
\end{equation}

Using the methodology from Section \ref{grks} we obtain a computable reconstruction
$\ter \in \fes_{r}^{t}(0,T; \fes_q^s).$

\begin{assumption}[Bounded reconstruction]\label{ass2}
 For the remainder of this section we will suppose that there is some compact and convex $\mathfrak{K} \subset \mathcal{U}$ such that
 \[ \ter(t,x) \in \mathfrak{K} \quad \forall (t,x) \in [0,T] \times \mathbb{T}.\]
\end{assumption}

\begin{remark}[Bounded reconstruction]
 Note that Assumption \ref{ass2} is verifiable in an a posteriori fashion, since $\ter$ is explicitly computable.
 It is, however, not sufficient to verify $\bu_h^n(x) \in \mathfrak{K}$ for all $n = 0, \dots, N$ and all $x \in \mathbb{T}.$
\end{remark}

\begin{remark}[Bounds on flux and entropy]
 Due to the regularity of $\bg$ and $\eta$ and the compactness of $\mathfrak{K}$ there exist
  constants $0 < \Cfu < \infty$ and $0< \Cetad < \Cetau < \infty,$  which can be explicitly computed from $\mathfrak{K}$, $\bg$ and $\eta,$
  such
  that
  \begin{equation}
    \label{eq:consts}
\norm{{\bv}^T \operatorname{H} {\bf g}(\bu) \bv} 
    \leq \Cfu \norm{\bv}^2,    
    \quad 
 \Cetad 
    \norm{\bv}^2
    \leq 
   {\bv}^T
    \operatorname{H} \eta(\bu)
    \bv
    \leq \Cetau \norm{\bv}^2
 \ \forall \ \bv \in \mathbb{R}^m, \bu \in \mathfrak{K},
  \end{equation}
  where $\norm{\cdot}$ is the Euclidean norm for vectors and $\operatorname{H}$ 
  denotes Hessian matrices.
\end{remark}

We can now define as in the previous section the temporal residual
\begin{equation}\label{sder}
{\bf R}^t := \partial_t \ter + {\boldsymbol f} (\ter) 
\end{equation}
with ${\bf R}^t \in \leb{2}((0,T);\fes^s_q).$ As $\ter$ is explicitly computable, so is ${\bf R}^t.$
Note that for $r\geq 3$ we even have ${\bf R}^t \in C^0((0,T);\fes^s_q).$

Our spatial reconstruction of $\ter$ is based on \cite{GMP_15}.
To this end we restrict ourselves to two types of numerical fluxes:

\begin{assumption}[Condition on the numerical flux]\label{ass1}
We assume that there exists a locally Lipschitz continuous function
$\bw: \mathcal{U} \times \mathcal{U} \rightarrow \mathcal{U}$ 
such that for any compact $K \subset \mathcal{U}$ there exists a constant $C_w(K) >0$
with
 \begin{equation}\label{w-cond}
  |\bw({\bf a}, {\bf b}) - {\bf a} | + |\bw({\bf a}, {\bf b}) - {\bf b} | \leq 
     C_w(K)|{\bf a} - {\bf b}| \quad \forall \ {\bf a}, {\bf b} \in K.
 \end{equation}
With this function the numerical flux $\bG$ is of one of the two following types:
 \begin{itemize}
  \item[(i)] 
 $
  \bG({\bf a}, {\bf b}) =\bg(\bw({\bf a}, {\bf b})) \quad \forall \ {\bf a}, {\bf b} \in \mathcal{U};
 $
 \item [(ii)] 
 $
  \bG({\bf a}, {\bf b}) =\bg(\bw({\bf a}, {\bf b}))  - \mu({\bf a},{\bf b};h) 
                         h^\nu ({\bf b} - {\bf a}) \quad \forall \ {\bf a}, {\bf b} \in \mathcal{U}
 $\\
 for some $\nu \in \mathbb{N}_0$ and matrix-valued function $\mu$ for which for any
 compact $K \subset \mathcal{U}$ there exists a constant $\mu_K > 0$ so that
   $|\mu({\bf a},{\bf b};h)| \leq \mu_K\left(1+\frac{|{\bf b}-{\bf a}|}{h}\right)$
   for $h$ small enough.
  \end{itemize}

\end{assumption}

\begin{remark}[Restrictions on the numerical flux] \label{rem:fluxes}
 \begin{enumerate}
  \item The condition imposed in  Assumption \ref{ass1} is stronger than the classical Lipschitz and consistency conditions.
  \item The conditions do not, by any means, guarantee stability of the numerical scheme. 
  In practical computations, interest focuses on numerical fluxes which satisfy one of the assumptions and lead to a stable numerical scheme, for obvious reason.
  \item The Lax-Wendroff and Richtmyer numerical fluxes, e.g.,  
     \begin{equation}\label{def:LW}
       \bG({\bf a},{\bf b})= \bg(\bw({\bf a},{\bf b})), 
         \quad \bw({\bf a},{\bf b})= \frac{{\bf a} + {\bf b}}{2} - \frac{\lambda}{2} ( \bg({\bf b}) - \bg({\bf a})),
     \end{equation}
    satisfy Assumption \ref{ass1} (i).
  \item The Lax Friedrichs flux 
    \begin{equation}\label{def:LLF}
      \bG({\bf a}, {\bf b}) =\frac{1}{2} \Big(
          \bg({\bf a}) + \bg({\bf b}) \Big)  - \lambda ({\bf b} - {\bf a})
    \end{equation}
    satisfies Assumption \ref{ass1} (ii) with 
    $\bw({\bf a},{\bf b})=\frac{1}{2}({\bf a}+{\bf b})$,
    $\nu=0$, and
    $$ \mu({\bf a},{\bf b},h)=\frac{\bg({\bf a}) - 2\bg(\bw({\bf a},{\bf b})) + \bg({\bf b})}
                     {2\|{\bf b}-{\bf a}\|^2} \otimes ({\bf b}-{\bf a}) - \lambda \mathbb{I}~,$$
                     where $\mathbb{I}$ denotes the $m \times m$ identity matrix.
  \item A Lax-Wendroff type flux with artificial viscosity as described for
    example in \cite[Ch. 16]{LeV02} satisfies Assumption \ref{ass1} (ii) with $\nu=1$. 
    In fact, the analysis presented in the following applies to a large class of general
    artificial viscosity methods typically used with Lax-Wendroff type fluxes, e.g., a discrete form of 
    $h^2\partial_x\big(|\partial_x u|\partial_x u\big),$ see
    \cite[e.g.]{Lapidus:1967:ArtVisc,Reisner:2013:ArtVisc}. In the numerical examples we
    will use a discrete version of this viscosity given by $\nu=1,\mu({\bf a},{\bf b};h)=\mu_0|{\bf b}-{\bf a}|/h$.
  \item The assumptions given here cover to a certain extent methods combining
    central with lower order upwind type fluxes, e.g.,
    $$ \bG({\bf a},{\bf b})= (1-\phi(h))\bg(\bw({\bf a},{\bf b})) + 
    \phi(h) \Big(\frac{1}{2}\big(\bg({\bf a})+\bg({\bf b})\big) - \lambda ({\bf b}-{\bf a})\Big) $$
    as long as $\phi(h)=O(h)$ for smooth solutions
    \cite[e.g.]{Toro:2000:FLIC,Qiu:2006:DGFLUXES}. To see this take $\nu=1$
    and 
    $$ \mu({\bf a},{\bf b};h) = \frac{\phi(h)}{2h} 
    \frac{\bg({\bf a})-2\bg(\bw({\bf a},{\bf b}))+\bg({\bf b})}
           {\|{\bf b}-{\bf a}\|^2}\otimes ({\bf b}-{\bf a}) - \lambda \mathbb{I}
    $$
 \end{enumerate}
\end{remark}

The spatial reconstruction approach is applied to $\ter(t,\cdot)$ for each
$t \in (0,T$ using the function $\bw$ to obtain a continuous
reconstruction $\str$:
\begin{definition}[Space-time reconstruction]\label{def:str}
Let $\ter$ be the temporal reconstruction as in Definition \ref{def:grec}
of a sequence $\{ \bu_h^n\}_{n=0}^N \subset \fes_q^s$ computed from a method of lines scheme consisting of a DG discretization in space 
using a numerical flux satisfying Assumption \ref{ass1}
and some single- or multi-step method in time.
Then, the space-time reconstruction
$\str(t,\cdot) \in \fes_{q+1}^s$ is defined by requiring
\begin{equation}\label{srec}
\begin{split}
 \int_{\mathbb{T}} (\str(t,\cdot) - \ter(t,\cdot)) \cdot {\boldsymbol \psi} &= 0 \quad \forall {\boldsymbol \psi} \in \fes_{q-1}^s\\
 \str(t, x_k^\pm)&= \bw(\ter(t,x_k^-),\ter(t,x_k^+))\quad \forall k.
 \end{split}
\end{equation}
\end{definition}

\begin{remark}[Choice of reconstruction]
\begin{enumerate}
\item We believe this reconstruction to be the most meaningful choice for general systems and fluxes as in Assumption \ref{ass1} (i) and (ii).
We will see in Lemma \ref{lem:compest3} that this choice may lead to a suboptimal residual in case $\nu=0.$
We cannot rule out the possibility that another reconstruction choice exists which leads to an estimator of optimal order for this type of flux also in the case $\nu=0.$
\item We will see in Section \ref{subs:LFs} that for strictly hyperbolic
  systems it is indeed possible to obtain an estimator of optimal order
  with a Roe type flux. But how to extend this construction to more
  general fluxes. We will
  for example demonstrate the difficulties in the case of a Lax-Friedrichs
  type flux.
\end{enumerate}
\end{remark}

\begin{lemma}[Properties of space-time reconstruction]
 Let  $\str$ be the spatial reconstruction defined by \eqref{srec}. Then, for each $t \in [0,T]$
 the function $\str(t,\cdot)$ is well-defined and locally computable.
 Moreover,
 \[ \str \in \sob{1}{\infty}(0,T;\fes_{q+1}^s \cap C^0(\mathbb{T})).\]
 As $\str$ is piece-wise polynomial and continuous in space it is also Lipschitz continuous in space.
\end{lemma}

\begin{proof}
 The facts that for each $t \in [0,T]$ the function
 $\str(t,\cdot)$ is well-defined, locally computable and continuous follow from \cite[Lem. 4.3]{GMP_15}.
 Assumptions \ref{ass1} and \ref{ass2} imply that $\bw$ is Lipschitz on the set of values taken by $\ter.$
 Thus, the Lipschitz continuity of $\ter$ translates into Lipschitz continuity of $\str$ in time.
\end{proof}


Since $\str$ is computable and Lipschitz continuous in space and time we may define
a computable residual
\begin{equation}\label{pclaw}
\bR^{st} :=  \partial_t \str + \partial_x {\bf g}(\str)
\end{equation}
with $\bR^{st} \in \leb{2}((0,T)\times \mathbb{T}, \mathbb{R}^m).$

We can now formulate our main a posteriori estimate:

\begin{theorem}
[A posteriori error bound]\label{The:re}
  Let $\bg \in C^2(\mathcal{U},\mathbb{R}^m)$  
  and let $\bu$ be an entropy solution  of
  \eqref{claw} with periodic boundary conditions.
   Let $\str$ be the space-time reconstruction of a fully discrete DG scheme
   $\{\bu_h^n\}_{n=0}^N$ defined according to Definition \ref{def:str}.
  Provided $\bu$ takes only values in $\mathfrak{K}$, then for $n=0,\dots,N$ the error between the numerical solution $\bu_h^n$ and $\bu(t_n,\cdot)$
  satisfies
  \begin{equation}\label{eq:there}
    \begin{split}
      \Norm{\bu(t_n,\cdot) - \bu_h^n}_{\leb{2}(\mathbb{T})}^2
      &\leq 
      2\Norm{\str (t_n,\cdot) - \bu_h^n}_{\leb{2}(\mathbb{T})}^2
      \\
      &\qquad
      +
      2\Cetad^{-1}
      \Big( 
      \Norm{\bR^{st}}_{\leb{2}((0,t_n)\times \mathbb{T})}^2 
      +
      \Cetau \Norm{\bu_0 - \str(0,\cdot)}_{\leb{2}(\mathbb{T})}^2 
      \Big)  
      \\
      &\qquad\qquad \times
      \exp\left(
        \int_0^{t_n}
        \frac{\Cetau \Cfu \Norm{\pdx \str(s,\cdot)}_{\leb{\infty}(\mathbb{T})} 
          + \Cetau^2}{\Cetad}  
        \d s
      \right)
      .    
    \end{split}
  \end{equation}
\end{theorem}
\begin{proof}
The theorem follows from \eqref{pclaw} in exactly the same way \cite[Thrm. 5.5]{GMP_15} follows from \cite[Eq. (5.2)]{GMP_15}.
\end{proof}

\begin{remark}[Computation of the estimator]\label{rem:l2}
 Note that if $\bw$ is not any smoother than Lipschitz continuous,
 $\str|_{(t_n,t_{n+1})}$ is also only Lipschitz continuous in time, while
 $\ter|_{(t_n,t_{n+1})}$ is polynomial.
 Since in this case the evaluation of $\Norm{\bR^{st}}_{\leb{2}((0,t_n)\times \mathbb{T})}^2 $
 with high precision is numerically extremely expensive, we aim to find smooth choices
 for $\bw$ in our test.
\end{remark}

\begin{remark}[Discontinuous entropy solutions]\label{rem:des}
\begin{enumerate}
 \item The reader may note that the estimate in Theorem \ref{The:re} does not require the entropy solution $\bu$ to be continuous.
  However, in case $\bu$ is discontinuous $\Norm{\pdx \str(s,\cdot)}_{\leb{\infty}(\mathbb{T})}$ is expected to scale like $h^{-1}$.
  Therefore, the estimator in \eqref{eq:there} will (at best) scale like $h^{q+1} \exp(h^{-1}),$ which diverges for $h \rightarrow 0.$
So in particular, the estimator diverges for $h \rightarrow 0,$ if the entropy solution is discontinuous.
\item The fact that the estimator does not converge to zero for $h, \tau
\rightarrow 0$ in case of a discontinuous entropy solution results from using the relative entropy framework, which does not guarantee uniqueness of
an entropy solution.
Indeed, it was shown in \cite{DS10} that, in general, entropy solutions may be non-unique for the Euler equations in several space dimensions.
\item It is well known that for nonlinear problems the
DG method is not stable in case of discontinuities in the solution. Thus we
can not expect the error to converge to zero and therefore the estimator
can not be expected to converge either. Stabilizing limiters can be included into the
framework but that is outside the scope of this paper and will be investigated
in a following study.
\end{enumerate}
\end{remark}

  Note that, adding zero, we may combine \eqref{sder} and \eqref{pclaw} in order to obtain
   \begin{equation}\label{stres}
\bR^{st} :=  \partial_t( \str - \ter) + \partial_x {\bf g}(\str) - \bfl(\ter) + \bR^t =: \bR^s + \bR^t.  
  \end{equation}
In this way we may decompose the residual into a 'spatial' and a 'temporal' part.
The analysis of \cite{ZS06} shows that the spatial part of the error of Runge-Kutta discontinuous Galerkin discretizations 
of systems of hyperbolic conservation laws is $\cO(h^{q+\gamma})$ with $q$ the polynomial degree of the DG scheme and $\gamma$ depending on the numerical flux. 
For general monotone fluxes  $\gamma=\tfrac{1}{2}$ and for upwind type fluxes $\gamma$ is improved to $1.$
When stating our optimality result below, we will assume that the true error of our scheme is $\cO(h^{q+\gamma} + \tau^r).$
We first consider the case of fluxes satisfying Assumption \ref{ass1} (i).

\begin{theorem}[Optimality of residuals]\label{thm:compest}
Let a numerical scheme which is of order $r$ in time and uses $q$-th order DG in space with a numerical flux satisfying Assumption \ref{ass1} (i) be given.
Let the temporal and spatial mesh size comply with  a CFL-type restriction $\tau = \cO(h_{min}).$
Let the residual $\bR^{st}$ be defined by \eqref{pclaw} and \eqref{srec}. Then, it is of optimal order $\cO(h^{q+\gamma} + \tau^r)$ with $\gamma \in \{ \tfrac{1}{2},1\}$, provided the exact error is of this order, 
and
the {\it Lipschitz constant} $L$ of $\bfl$ defined in \eqref{dgscheme} in the sense of \eqref{Lip} behaves like $h_{\min}^{-1}.$
\end{theorem}
\begin{proof}
 Due to \eqref{stres} it is sufficient that $\bR^s$ and $\bR^t$ are of optimal order.
 The temporal residual $\bR^t$ is of the type of residuals investigated in Section \ref{sec:ode}.
Invoking Theorem \ref{thrm:ode} we obtain 
\begin{equation}
  \| \bR^t\|_{\leb{\infty}(0,T;\leb{2}(\mathbb{T}))} = \sum_k L^k \cO((\tau^r+ h^{q+\gamma})\tau^k).
\end{equation}
The condition on  $L$ and the CFL condition ensure that $L^k \tau^k = \cO(1)$ such that $\bR^t$ is of optimal order.
The spatial residual $\bR^s$ is of the form of residuals investigated in \cite{GMP_15} and can be estimated as in \cite[Lem. 6.2]{GMP_15} with 
$\bu_h$ being replaced by $\ter.$
Arguments similar to
\cite[Rem. 6.6]{GMP_15} show that $\bR^s$ is of optimal order.
\end{proof}

\begin{remark}[Spatial residual]
Note that any computable reconstruction $\tilde{\bu} \in \sob{1}{\infty}((0,T)\times \mathbb{T}, \mathcal{U})$ of the numerical solution gives rise to an error estimate
  of the form \eqref{eq:there}.
  Our particular reconstruction and (at the same time) the condition on the numerical flux in Assumption \ref{ass1} (i) is driven by our desire for the spatial part $\bR^s$ 
  of the residual to be of optimal order.
  This part of the residual, and its optimality, was extensively investigated in \cite{GMP_15}.
  We will investigate the effects of added artificial viscosity in Theorem
  \ref{lem:compest2} and Lemma \ref{lem:compest3} showing that low order viscosity leads to
  a suboptimal convergence of the estimator.
\end{remark}

Let us turn our attention to numerical fluxes satisfying Assumption \ref{ass1} (ii).
Our goal is to ascertain the effect of artificial viscosity on the order of the residual.
Note that we define the reconstruction by \eqref{srec} as before, accounting for $\bw$ but not for the artificial viscosity.
In Section \ref{subs:LFs} we will show for a Roe type flux that for one flux there might be different choices of $\bw$ leading 
to different values of $\nu$
It is not obvious whether this approach can be generalized so that optimal reconstructions can be obtained for more general numerical fluxes.

To simplify the presentation we will assume in the following that 
$\tau = \cO(h_{min})$ and that the order of the time stepping method is compatible
with the order of the space discretization, i.e., $r=q+1$. 
We will first show how an upper bound for the residual depends on the order
$\nu$ of the artificial viscosity.
\begin{theorem}[Conditional optimality of residuals]\label{lem:compest2}
For $q \geq 1$ let a numerical scheme which is order $q+1$ in time and uses $q$-th order DG in space with a numerical flux satisfying Assumption \ref{ass1} (ii) be given.
Let the residual $\bR^{st}$ be defined by \eqref{pclaw} and \eqref{srec}.
Let the exact error be of order
 $\cO(h^{q+\gamma}).$
 Then,  $\bR^{st}$ is of order $\cO(h^{q+\gamma} +h^{q+\gamma+\nu-1})$ with $\gamma \in \{ \tfrac{1}{2},1\}$, provided
the {\it Lipschitz constant} $L$ of $\bfl$ defined in \eqref{dgscheme}, in the sense of \eqref{Lip}, behaves like $h_{\min}^{-1}.$
\end{theorem}
\begin{remark}[Conditional optimality]
 Note that the rate proven here is optimal for $\nu \geq 1$ but suboptimal otherwise.
 We will show in Lemma \ref{lem:compest3} that the $\cO(h^{q+\gamma+\nu-1})$ part of the residual is actually present. 
\end{remark}

\begin{proof}
Our argument here is based on an observation in \cite{MN06} 
which states that
\begin{equation}\label{mn06}
  \sum_i h_i | \jump{\ter }_i |^2  = \cO(h^{2q +2\gamma}),
\end{equation}
i.e., this sum is of the order of the square of the $\leb{2}$-norm of the true error.
Let us note that \eqref{mn06} implies
\begin{multline}\label{jumpbound}
\max_i  \left|\frac{\jump{\ter}_i}{h_i} \right|^2  \lesssim \frac{1}{h^3} \max_i \Big( h_i \left|\jump{\ter}_i\right|^2\Big)
\lesssim \frac{1}{h^3} \sum_i h_i | \jump{\ter }_i |^2 \lesssim \cO(h^{2q +2\gamma-3})~.
\end{multline}
Since $q \geq 1$ equation \eqref{jumpbound} implies
\begin{equation}\label{lip2}
\max_i  \left|\frac{\jump{\ter}_i}{h_i} \right| \lesssim 1.
\end{equation}
We will not consider the full spatial residual $\bR^s,$ since,
following the arguments in the proof of \cite[Lem. 6.2]{GMP_15}, 
the only part which changes as Assumption \ref{ass1} (i) is replaced by \ref{ass1} (ii) is
the estimate of the $\leb{2}$-norm of 
\begin{equation}\label{gfub}
 {\bf R}_g:=\cP_q [\partial_x {\bf g}(\str)]-\bfl(\ter) \in \fes_q^s
\end{equation}
where $\cP_q$ denotes $\leb{2}$-orthogonal projection into $\fes_q^s$ and we note that $\bfl(\ter)$ corresponds to $\partial_x \hat {\bf g}$ in \cite{GMP_15}.
We will only consider this part \eqref{gfub} of the spatial residual in the sequel.
Using integration by parts we obtain 
\begin{equation}\label{sopt1}
 \begin{split}
 \int_{\mathcal{T}} |{\bf R}_g|^2\d x& =\int_{\mathcal{T}} (\cP_q [\partial_x {\bf g}(\str)]- \bfl(\ter))\cdot {\bf R}_g\d x \\
  &= \int_{\mathcal{T}} (\partial_x {\bf g}(\str)- \bfl(\ter))\cdot{\bf R}_g\d x
 \\
 &= \int_{\mathcal{T}} ({\bf g} (\ter) -{\bf g}(\str))  \cdot \partial_x^e {\bf R}_g \d x
 + \sum \Big({\bf G}((\ter)^\pm)- {\bf g}({\bf w}((\ter)^\pm))\Big)\cdot\jump{{\bf R}_g}\\
 &= \int_{\mathcal{T}} ({\bf g} (\ter) -{\bf g}(\str))  \cdot \partial_x^e {\bf R}_g\d x
 + h^\nu \sum \mu((\ter)^-,(\ter)^+;h_i) \jump{\ter}\cdot \jump{{\bf R}_g}\\
 &=: E_1 + E_2 ,
 \end{split}
\end{equation}
where $(\ter)^\pm$ is an abbreviation for $((\ter)^-,(\ter)^+):=(\ter(x_i^-),\ter(x_i^+)) .$
Using the same trick as in the proof of  \cite[Lem. 6.2]{GMP_15} we obtain 
\begin{multline}\label{e1}
 |E_1| \lesssim
    \norm{\ter}_{\sob{1}{\infty}(\mathcal{T})}
    \Norm{\str - \ter}_{\leb{2}(\mathbb{T})} 
    \Norm{{\bf R}_g}_{\leb{2}(\mathbb{T})}\\
    + 
    \left( \sum_{i=0}^{M-1} \frac{1}{h_{i-\frac{1}{2}}^2}   \int_{x_{i-1}}^{x_i} \norm{\str - \ter}^4\d x\right)^{\frac{1}{2}} \Norm{{\bf R}_g}_{\leb{2}(\mathbb{T})}.
\end{multline}
In order to bound $E_2$ we employ Cauchy-Schwarz inequality, trace inequalities \cite[Lem. 1.46]{DE12}, and \eqref{lip2} so that we get
\begin{multline}\label{e2}
  |E_2| \lesssim \mu_\mathfrak{K} h^\nu \Big( \sum_{i=0}^{M-1}
  \Big(1+\frac{\jump{\ter}}{h_i}\Big)^2\frac{1}{h_i} \jump{\ter}^2\Big)^{\frac{1}{2}}
  \Big( \sum_{i=0}^{M-1} h_i \jump{{\bf R}_g}^2\Big)^{\frac{1}{2}} 
 \\ 
 \lesssim \mu_\mathfrak{K} h^\nu \Big( \sum_{i=0}^{M-1} \frac{1}{h_i} \jump{\ter}^2\Big)^{\frac{1}{2}} \Norm{{\bf R}_g}_{\leb{2}(\mathbb{T})}.
\end{multline}
Inserting \eqref{e1} and \eqref{e2} into \eqref{sopt1} and dividing
 both sides by $ \Norm{{\bf R}_g}_{\leb{2}(\mathbb{T})}$ we obtain  
 \begin{multline}\label{rs}
  \Norm{{\bf R}_g}_{\leb{2}(\mathbb{T})} \leq\norm{\ter}_{\sob{1}{\infty}}
    \Norm{\str - \ter}_{\leb{2}(\mathbb{T})} 
    \\
    + 
    \left( \sum_{i=0}^{M-1} \frac{1}{h_{i-\frac{1}{2}}^2}   \int_{x_{i-1}}^{x_i} \norm{\str - \ter}^4\d x \right)^{\frac{1}{2}}
    +\mu_\mathfrak{K} h^\nu \Big( \sum_{i=0}^{M-1} \frac{1}{h_i} \jump{\ter}^2\Big)^{\frac{1}{2}}.
 \end{multline}
Following the arguments of  \cite[Rem. 6.6]{GMP_15} we see that the first two terms on the right hand side of \eqref{rs} are of order
 $\cO(h^{q+\gamma})$ 
while the last term, which has no counterpart in the analysis presented in  \cite{GMP_15}, is of order
$\cO(h^{\nu-1+q+\gamma}).$ 
\end{proof}


\subsection{Optimal reconstruction for a Roe-type flux}\label{subs:LFs}

The goal of this section is twofold. Firstly, we will study a Roe type flux with a
reconstruction based on a simple average showing that this can lead to suboptimal convergence
of the residual. The flux with this choice of ${\bf w}$ fits into the framework of Assumption \ref{ass1} (ii) with $\nu=0$ 
thus we show that our estimate in Theorem~\ref{lem:compest2} is sharp.
Secondly, we show that for the same flux it is possible to find a more
involved version of ${\bf w}$,
leading to a reconstruction with optimal order residual. We hope also to convince the reader of the 
difficulty of finding such a reconstruction for general numerical flux functions.

In the following we consider a strictly hyperbolic system of the form 
\eqref{claw} using a numerical flux of Roe type. In the following we fix two
elements ${\bf a},{\bf b}$ in a convex, compact subset $K \subset \mathcal{U}$ and denote their average with
${\bf c}:=\frac{1}{2}{({\bf a}+{\bf b})}$. Let $A:=Dg({\bf c})$ be the flux Jacobian
at this point. Due to the hyperbolicity of the system, $A$ is diagonalizable:
$LAR=D$ with a diagonal matrix $D={\rm diag}(\lambda_1,\dots,\lambda_m)$ consisting of
the eigenvalues and matrices $R$ containing the right Eigenvectors ${\bf r}_i$ as columns and
$L$ with left eigenvectors ${\bf l}_i$ as rows. We choose the right and the left eigenvectors to be dual to each other, i.e., ${\bf l}_i\cdot {\bf r}_j = \delta_{ij}$.
Consider now the numerical flux function of Roe type given by
\begin{equation}\label{lff} 
  \bG({\bf a},{\bf b})= \bg({\bf c}) + \frac{1}{2} \big| A \big| ({\bf a} -{\bf b})~,
\end{equation}
where $|A|=R\;|D|\;L$ and $|D|:={\rm diag}(|\lambda_1|,\dots,|\lambda_m|)$.
Note that, as we restrict ourselves to the strictly hyperbolic case, appropriately normalized
right and left eigenvalues are $C^{1}$ vector fields on $\mathcal{U}$ and, thus, there is a bound $C^*>0$ on
$| L | \cdot | R | $ depending on $K,$ but not on the specific choice of ${\bf a},{\bf b} \in K.$

First we will demonstrate that in general taking ${\bf w}({\bf a},{\bf b})={\bf c}$ does not lead
to an optimal error estimate. It is easy to see that this choice leads to a
$\mu$ which satisfies Assumption \ref{ass1} (ii) with $\nu=0$. 
Taking for example the scalar case and noting $|A|=|g'(c)|$ we find:
\begin{equation}
  |\mu(a,b;h)| = \frac{|g(c) - G(a,b)|}{|b-a|} = \frac{1}{2}|g'(c)|
\end{equation}
which is clearly bounded on any compact subset of the state space.
The following Lemma shows that the suboptimal rate stated in
Theorem~\ref{lem:compest2} is
sharp and that therefore optimal order for the residual is in general only
guaranteed if the artificial viscosity term is chosen with $\nu\geq 1$.
Taking this result together with the observation made above it is clear
that ${\bf w}({\bf a},{\bf b})={\bf c}$ will in general not lead to an optimal
rate of convergence of the residual.

\begin{lemma}[Suboptimality]\label{lem:compest3}
Consider the scalar linear problem
  \[ u_t +u_x =0.\]
and a numerical scheme which is order $2$ in time and uses first order DG
in space on an equidistant mesh of size $h$ with a numerical flux satisfying 
Assumption \ref{ass1} (ii) with $\nu=0$ and $\mu(a,b;h)=\mu_0 >0$.
Let the temporal and spatial mesh size comply with  a CFL-type restriction
$\tau = \cO(h).$
Then, the norm of the residual  $\bR^{st}$, defined by \eqref{pclaw} and \eqref{srec}, is bounded 
from below by terms of order
$h^{\gamma}$ even if the error of the method is $\cO(h^{1+\gamma}).$
\end{lemma}
\begin{proof}
 As argued in the proof of Theorem \ref{lem:compest2} it is sufficient to show that $\Norm{{\bf R}_g}_{\leb{2}(\mathbb{T})}$ is bounded 
 from below by terms of order $h^{\gamma}.$ All the other terms are of higher order and can, therefore, not cancel with ${\bf R}_g.$
 The assumptions of the Lemma at hand are a special case of those of
 Theorem \ref{lem:compest2} so that we have \eqref{sopt1}.
Using $g(u)=u$ and \eqref{srec}$_1$ we obtain, analogous to \eqref{sopt1},
\begin{equation}
\int_{\mathbb{T}} R_g \phi \operatorname{d} x=  \int_{\mathcal{T}} (\cP_1 [\partial_x  g(\hat u^{st})]- f(\hat u^t)) \phi \operatorname{d} x
 = 
  \sum \mu_0 \jump{\hat u^t} \jump{\phi}~,
\end{equation}
for all $\phi \in \fes^s_q.$
Since, an orthonormal basis of $\fes_1^s$ is given by $\{ \phi_j , \psi_j\}_{j=0}^{M-1}$
with 
\begin{equation}
 \begin{split}
  \phi_j (x)&:= \left\{ \begin{array}{ccc}
                       \sqrt{\frac{1}{2 h}}&:& x \in [x_j,x_{j+1}] \\ 0 &:& \text{else}
                      \end{array}
  \right.\\
  \psi_j(x)&:= \left\{ \begin{array}{ccc}
                       \sqrt{\frac{3}{2h}}( x- \frac{x_j+x_{j+1}}{2}) &:& x \in [x_j,x_{j+1}] \\ 0 &:& \text{else}
                      \end{array}
  \right.
 \end{split}
\end{equation}
we get for any fixed $t \in [0,T]$
\begin{equation}\label{lflb}
\begin{split}
 &\Norm{\cP_1 [\partial_x  g(\hat u^{st})]- f(\hat u^t)}_{\leb{2}}^2\\
 &=  \Big( \int_{\mathcal{T}}  \Big(\cP_1 [\partial_x  g(\hat u^{st})]- f(\hat u^t) \Big)\phi_j\Big)^2
 + \Big( \int_{\mathcal{T}} \Big( \cP_1 [\partial_x  g(\hat u^{st})]- f(\hat u^t)\Big) \psi_j\Big)^2\\
&=  \mu_0 \sum_j \left( \frac{1}{2h}\big(-\jump{ \hat u^t}_j +  \jump{\hat u^t }_{j+1}\big)^2 
+ \frac{3}{2h}\big(\jump{ \hat u^t}_j +  \jump{ \hat u^t }_{j+1}\big)^2  \right)\\
&=  \frac{\mu_0}{2h} 
\sum_j \left( 4 \big( \jump{\hat u^t }_{j+1}\big)^2 + 4\jump{\hat u^t }_{j+1}\jump{\hat u^t }_{j} + 4 \big( \jump{\hat u^t }_{j}\big)^2 \right)  \\
&\geq  \frac{\mu_0}{2h} 
\sum_j \left( 2 \big( \jump{\hat u^t }_{j+1}\big)^2 +  2 \big( \jump{\hat u^t }_{j}\big)^2 \right)  = 
\frac{2\mu_0}{h}
\sum_j \big( \jump{\hat u^t }_{j}\big)^2 .
 \end{split}
\end{equation}
According to the arguments given in \cite[Rem. 6.6]{GMP_15} the lower bound derived in \eqref{lflb} is of order $h^{\gamma}$ even if the error of the method is $\cO(h^{1+\gamma}).$
\end{proof}

We conclude this section by showing that we can choose a function $\bw({\bf a},{\bf b})$ such that \eqref{lff}
satisfies Assumption \ref{ass1} (ii) with $\nu=1$.
Using Theorem~\ref{lem:compest2} we thus obtain a reconstruction of optimal order.
We restrict ourselves to a DG scheme with polynomials of degree $q \geq 1$ such that, 
following the discussion in the proof of Theorem \ref{lem:compest2}, we can
restrict our attention to the case
\begin{equation}\label{lesssim}
 \big|{\bf a} - {\bf b}\big| \lesssim h.
\end{equation}

In order to define the reconstruction we define the characteristic decomposition of ${\bf a},{\bf b}$ given by
${\boldsymbol \alpha}:=L{\bf a},{\boldsymbol \beta}:=L{\bf b}$ and will study a $\bw$ of the form
$\bw({\bf a},{\bf b}):=R{\boldsymbol \omega}({\boldsymbol \alpha},{\boldsymbol \beta})$. To define ${\boldsymbol \omega}=\big(\omega_i(\alpha_i,\beta_i)\big)_{i=1}^m$
consider a (not-strictly) monotone smooth function $\chi : \mathbb{R} \rightarrow \mathbb{R}$ such that 
$\chi(z)=0$ for $z<-1$ and $\chi(z)=1$ for $z>1.$ We define
$\chi_h(z):= \chi(z/h).$ Then
\begin{equation}\label{lfur} 
  \omega_i(\alpha_i,\beta_i) = \chi_h\left( \lambda_i \right)\;\alpha_i +  
           \Big( 1 - \chi_h\left(\lambda_i\right)\Big)\;\beta_i~.
\end{equation}
Note that ${\boldsymbol \omega}$ provides upwinding of the characteristic variables smoothed out so that the
reconstruction is smooth enough to allow for an efficient computation of
the integrated residual. 

Finally, we can now define the function $\mu$ 
\[ \mu({\bf a},{\bf b};h) = \frac{\bg(\bw({\bf a},{\bf b})) - \bG({\bf a},{\bf b})}{h \| {\bf a}- {\bf b}\|^2} \otimes ({\bf b}-{\bf a}).\]
In order to show that $\mu$ can be bounded such that 
Assumption \ref{ass1} (ii) is satisfied with $\nu=1$, we need to prove 
\begin{equation}\label{a:goal}
   \big| \bg(\bw({\bf a},{\bf b})) -\bG({\bf a},{\bf b})\big| \stackrel{!}{\lesssim} h | {\bf b}- {\bf a}|.
   \end{equation}
   Using Taylor expansion we obtain 
\begin{multline*}
   \big| \bg(\bw({\bf a},{\bf b})) -\bG({\bf a},{\bf b})\big| \leq
  \big| A(\bw({\bf a},{\bf b})-{\bf c}) - \frac{1}{2}|A|({\bf a}- {\bf b}) \big|\\+ \big|(\bw({\bf a},{\bf b})-{\bf c})^T\operatorname{H}\bg({\boldsymbol \xi})(\bw({\bf a},{\bf b})-{\bf c}) \big| =: S_1 + S_2
  \end{multline*}
  where $\operatorname{H}\bg$ is a tensor of third order consisting of Hessians of the components of $\bg$ and ${\boldsymbol \xi}$ is a convex combination of $\bw({\bf a},{\bf b})$ and ${\bf c}$.

Defining ${\boldsymbol \gamma} = L{\bf c} = \frac{1}{2}({\boldsymbol \alpha}+{\boldsymbol \beta})$
we can bound $S_1$ by
\begin{equation}
    S_1 = \big| R D({\boldsymbol \omega}({\boldsymbol \alpha},{\boldsymbol \beta})-
          {\boldsymbol \gamma}) - \frac{1}{2}R|D|({\boldsymbol \alpha}-{\boldsymbol \beta})  \big|
          \leq |R| \ | {\boldsymbol \delta} |
\end{equation}
with 
${\boldsymbol \delta}:=D({\boldsymbol \omega}({\boldsymbol \alpha},{\boldsymbol \beta})-{\boldsymbol \gamma}) - 
   \frac{1}{2}|D|({\boldsymbol \alpha}-{\boldsymbol \beta}) $.

Now now consider each component of the vector ${\boldsymbol \delta}$ separately,
distinguishing two cases:

\subsection*{ Case one: $\big|\lambda_i| \geq h$.\\}
We show the computation for $\lambda_i \geq h, $ and $ \lambda_i \leq - h $ is analogous:
\begin{multline} 
\lambda_i\big(\omega_i(\alpha_i,\beta_i)-\gamma_i\big) -
         \frac{1}{2}|\lambda_i|(\alpha_i-\beta_i) = 
 \lambda_i(\alpha_i-\gamma_i) - \frac{1}{2}\lambda_i(\alpha_i-\beta_i) = 0
\end{multline}
\subsection*{
Case two: $-h \leq \lambda_i \leq  h$.\\}
Now $ 
\big|\lambda_i \big(\omega_i(\alpha_i,\beta_i)-\gamma_i\big) -
        \frac{1}{2} |\lambda_i|(\alpha_i-\beta_i) \big| \leq
  h |\alpha_i-\beta_i|.
$

Combining both cases we get that each component of ${\boldsymbol \delta}$
can be bounded by $h|\alpha_i-\beta_i|$ and thus
$$ S_1 \leq h|{\boldsymbol \alpha}-{\boldsymbol \beta}| \lesssim h |{\bf a}-{\bf b}|. $$

In order to bound $S_2$ we observe
\begin{multline}\label{wmc}
S_2\leq \big| \operatorname{H}\bg({\boldsymbol \xi}) \big| \big|\bw({\bf a},{\bf b})-{\bf c}\big|^2
 \leq  \big| \operatorname{H}\bg({\boldsymbol \xi}) \big| \, | R|^2 \, \big| {\boldsymbol \omega}({\boldsymbol \alpha},{\boldsymbol \beta})-{\boldsymbol \gamma} \big|^2\\
 \leq \big| \operatorname{H}\bg({\boldsymbol \xi}) \big|\, | R|^2 \, \big|{\boldsymbol \alpha} - {\boldsymbol \beta}\big|^2
  \lesssim  \big| \operatorname{H}\bg({\boldsymbol \xi}) \big| h \big|{\bf a} - {\bf b}\big|,
\end{multline}
so that it remains to show that $| \operatorname{H} \bg({\boldsymbol \xi})|$ is bounded uniformly in $h$
for $h$ small enough.
It is sufficient to show that ${\boldsymbol \xi}$ is in some compact subset of $\mathcal{U}$ for $h$ small enough.
Since $K$ is compact and $\mathcal{U}$ is open, there exists $\varepsilon>0$ such that
$K_\varepsilon := \{ x \in \mathbb{R}^m \, : \, \operatorname{dist}(x,K) \leq \varepsilon\}$
is a convex, compact subset of $\mathcal{U}$.
By \eqref{wmc} and \eqref{lesssim} we know  ${\bf c} \in K$ and $| \bw -{\bf c}| \lesssim h$ so that
$\bw \in K_\varepsilon$ for $h$ small enough. This implies ${\boldsymbol \xi} \in  K_\varepsilon$
which completes the proof of \eqref{a:goal}. 
\begin{remark}[Difficulty for general fluxes]
For the Roe type flux studied here an optimal reconstruction is possible
since the artificial viscosity term $|A|({\bf a}-{\bf b})$
matches the linear term in the Taylor expansion of $\bg$ so that the vector 
${\boldsymbol \delta}$ vanishes in the first case studied above. 
In the second case it is crucial that we have the same Eigenvalue in both
terms contributing to $\delta_i$ so that both are proportional to 
$h|{\bf a}-{\bf b}|$. Take for example a viscosity term typically used in
local Lax-Friedrichs type fluxes, i.e., of the form
$\max_k|\lambda_k|({\bf a}-{\bf b})$.
Now consider a situation with $\lambda_1=0$ then
$|\delta_1|= \max_k|\lambda_k|\;|\alpha_1-\beta_1|$.
Now if $\alpha_1\neq\beta_1$ and for example $\lambda_2=1$ then
$\delta_1$ can only be bounded by $|{\bf a}-{\bf b}|$ without the 
factor of $h$ which would be required for an optimal reconstruction.
\end{remark}


\section{Numerical experiments}\label{sec:num}

In the following we will show some numerical tests verifying the convergence
rates presented in the previous section. As pointed out smooth
reconstructions are preferable to allow for an efficient computation of the
integrated residual. One numerical flux which gives good results as 
long as the entropy solution is quite regular and allows for optimal
estimates with a smooth $\bw$ is the Richtmyer flux \eqref{def:LW} with or
without additional diffusion term. We have performed tests using the
following numerical flux
 \begin{equation}\label{def:Ri}
   \bG({\bf a},{\bf b})= \bg(\bw({\bf a},{\bf b})) - \mu |{\bf b}-{\bf a}|({\bf b}-{\bf a})~,
     \quad \bw({\bf a},{\bf b})= \frac{{\bf a} + {\bf b}}{2} -
     \frac{\lambda}{2} ( \bg({\bf b}) - \bg({\bf a}))~.
 \end{equation}
Where $\lambda=\frac{\tau}{h}$ and $\mu=\frac{1}{2}$ or $\mu=0$. The artificial
viscosity provides a simple approximation of a viscosity of the form
$h^2\partial_x\big(|\partial_x u|\partial_x u\big)$.
Thus we are either in the case of 
Assumption \ref{ass1} (ii) ($\mu=\frac{1}{2}$) with $\nu=1$ or in the case
of Assumption \ref{ass1} (i) ($\mu=0$) so that we can expect optimal
convergence of the residuals. 
We studied both $\mu=0$ and $\mu=\frac{1}{2}$ in \eqref{def:Ri} but found
the difference to be negligible so that we will only show results for one
of those choices in the following.

If not stated otherwise we use an explicit Runge-Kutta method of order $r$
and a reconstruction in time based on the Hermite polynomial $H(p,0,0)$ or
$H(p,0,-1)$ with $r,p$ chosen to match the rate of the spatial scheme.
In all the following figures we study both the norm of the residual and the
error. The figures on the left show the values using a logarithmic scale
and the corresponding experimental orders of convergence are shown in the
right plot. Note that from the estimate \eqref{eq:there} we focus only on
the term $\Norm{\bR^{st}}_{\leb{2}((0,t_n)\times \mathbb{T})}$ involving the residual
but will ignore the exponential factor and error in the initial conditions.
Also in all our tests the first term in the estimator involving the error of the
spatial reconstruction was of the same order as the residual. 

\subsection{Linear problem}
Consider the linear scalar conservation law
$$ \partial_t u + 8\partial_x u = 0 $$
in the domain $[0,2]$ with periodic boundary conditions and initial
condition $u(0,x) = 1-\frac{1}{2}\cos(\pi x)$. The problem is solved up to
time $T=0.4$. We compute the experimental
order of convergence starting our
simulations with $h=0.125$ and $\tau=0.02$ and reducing both $h$ and $\tau$
by a factor of $2$ in each step.  This leads to a CFL
constant of about $0.13$ which is sufficiently small for all polynomial degrees used in
the following simulations. 

In Figure \ref{fig:advModLW} we show the values and experimental order of
convergence of the error and the residual for different polynomial degrees.
It is evident that for all polynomial degrees $q=1,\dots,3$
both the error and the residual converge with the expected order of
$q+1$. In Figure \ref{fig:advLLF} we use a Lax-Friedrichs numerical flux function of the form 
\eqref{def:LLF}. For both $q=2$ and $q=3$ the suboptimal convergence of the residual
predicted by the theory presented in the previous section is evident.
Finally we investigate the influence of the temporal reconstruction on the
convergence of the residual in Figure~\ref{fig:advtrecon}. We use a Hermite
interpolation of one degree lower then our theory demands. For $q=2$ the
quadratic reconstruction is clearly not sufficient for optimal convergence
of the residual. For $q=3$ the results are not conclusive since the residual
is still converging with order $4$ for this simple test case.

\begin{figure}
\includegraphics[width=\textwidth]{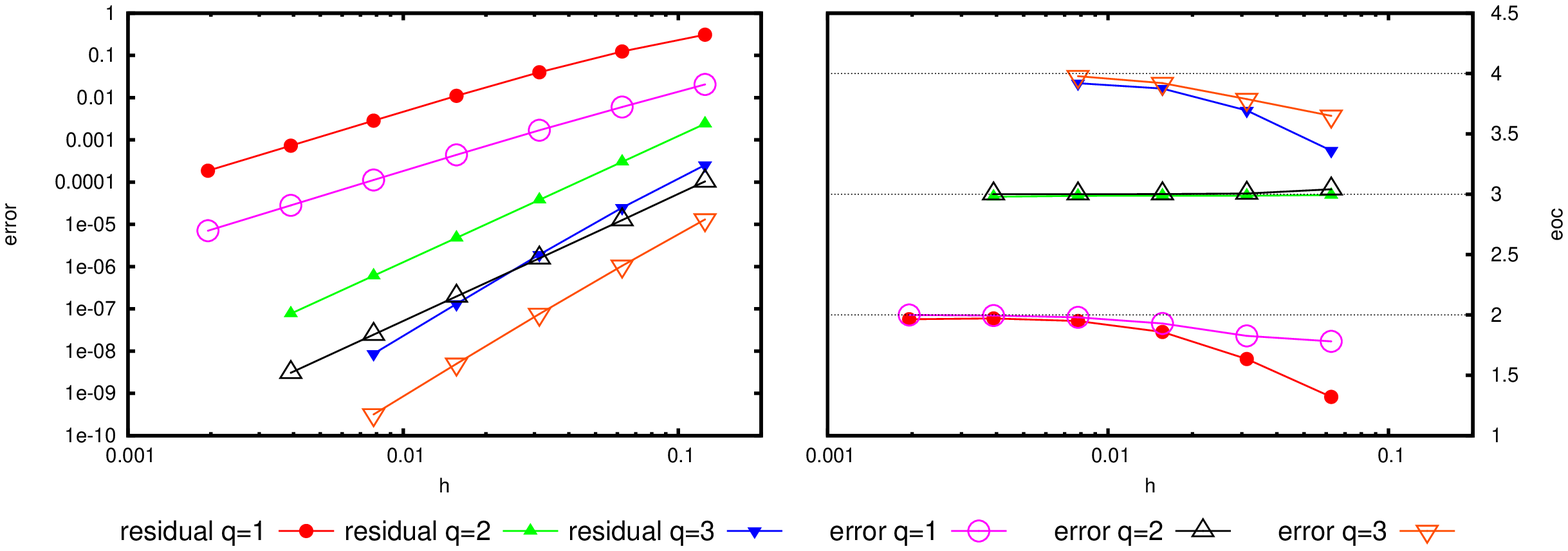}
\caption{Error and residuals for linear advection problem and polynomial
degree $q=1,\dots,3$ using flux \eqref{def:Ri} with $\mu=\frac{1}{2}$.}
\label{fig:advModLW}
\end{figure}

\begin{figure}
\includegraphics[width=\textwidth]{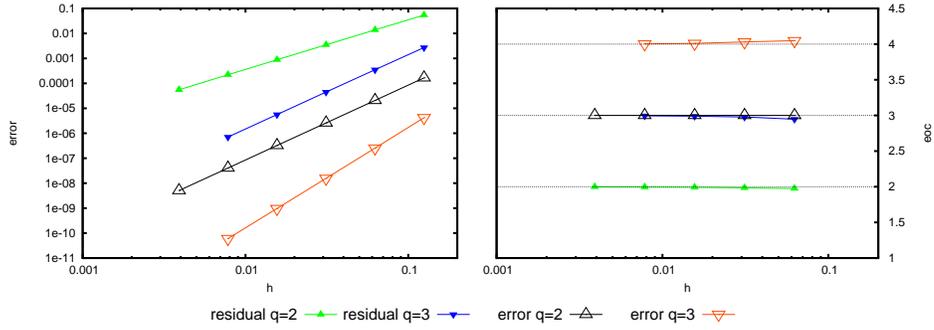}
\caption{Error and residuals for linear advection problem and polynomial
degree $q=2,3$ using a local Lax-Friedrichs type flux.}
\label{fig:advLLF}
\end{figure}

\begin{figure}
\includegraphics[width=\textwidth]{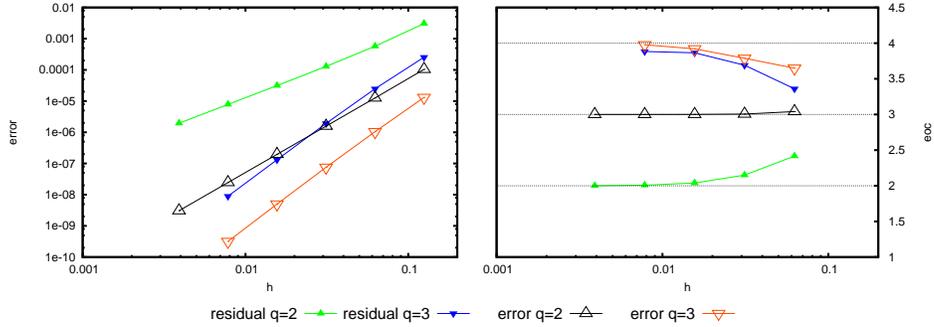}
\caption{Error and residuals for linear advection problem and polynomial
degree $q=2,3$ using flux \eqref{def:Ri} with $\mu=\frac{1}{2}$ but a
temporal reconstruction of lower order,}
\label{fig:advtrecon}
\end{figure}

\subsection{Euler equation}

We conclude our numerical experiments with some tests using the
compressible Euler equations of gas dynamics with an ideal pressure law
with adiabatic constant $\gamma=1.4$. We use the same domain and initial
grid as in the previous test case. The time step on the coarsest grid is
set to be $\tau=0.008$.
The initial conditions consist of a constant density and
velocity of $\rho=1,u=1$ and a sinusoidal pressure wave
$p(x)=1.3+\frac{1}{2}\sin(\pi x)$. We again use the flux~\eqref{def:Ri}
with $\mu=0$. In Figure~\ref{fig:eulerEvolution} a few snapshots
in time of the density are displayed. 
Note that we do not have an exact solution in this case and therefore we can not compute the
convergence rate of the scheme directly. But since the solution clearly
remains smooth up to $t=1$ we can expect our residual to converge with
optimal order up to this point in time. This is confirmed by the results shown in
Figure~\ref{fig:eulerLW}. 

\begin{figure}
\includegraphics[width=\textwidth]{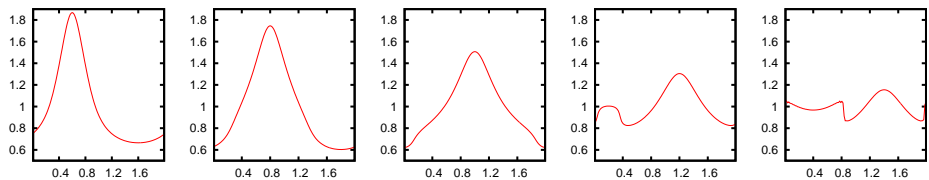}
\caption{Time evolution of the density for the Euler test case. From left
to right: $t=0.6,0.8,1.0,1.2,1.4$.}
\label{fig:eulerEvolution}
\end{figure}

\begin{figure}
\includegraphics[width=\textwidth]{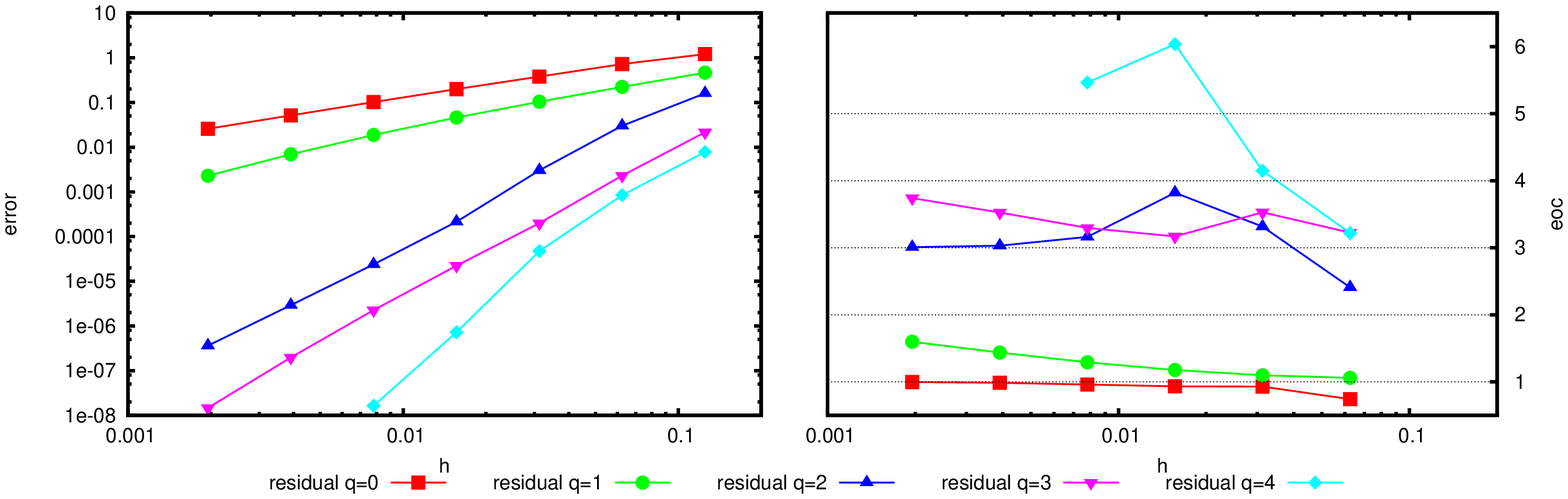}
\caption{Error and residuals for linear advection problem and polynomial
degree $q=0,\dots,4$ using flux \eqref{def:Ri} with $\mu=0$.}
\label{fig:eulerLW}
\end{figure}

\section{Conclusions}

In this work we extended previous results on a posteriori error estimation
based on relative entropy to a fully discrete scheme. 
The resulting estimator requires the computation of first a temporal and 
then a spatial reconstruction of the solution. 
The temporal reconstruction is based on a Hermite interpolation and is independent of
the time stepping method used. 
The spatial reconstruction uses the idea presented in
\cite{GMP_15}.  As noted there the reconstruction and the
flux have to be chosen carefully to guarantee optimal convergence of the
scheme. We managed to generalize the assumptions on the numerical flux 
greatly extending the class of schemes for which we can prove optimal convergence of the estimator.
This extended class of schemes now include Lax-Wendroff type flux with artificial diffusion
and we have shown that our extensions cover a Roe flux if upwinding in
the characteristic variables is used to define the reconstruction - simple
averaging on the other hand leads to suboptimal convergence as we have
shown. So we have managed to prove that in general the
requirements stated are not only sufficient but in fact
necessary and have also demonstrated numerically that otherwise the residual will converge
with an order smaller than the scheme.

In future work we plan to extend the results to higher space dimensions.
We will also investigate how the residual can be used to drive grid adaptation and
possibly to detect regions of discontinuities in the solution required in
the design of stabilization methods. These could either be based on local
artificial diffusion or on limiters typically used together with DG
methods. 

\bibliographystyle{alpha}
\bibliography{nskbib}

\newcommand{\etalchar}[1]{$^{#1}$}
\def\cprime{$'$}
\begin{thebibliography}{EGGH98}

\bibitem[AAB{\etalchar{+}}11]{AABCP11}
R{\'e}mi Abgrall, Denise Aregba, Christophe Berthon, Manuel Castro, and Carlos
  Par{\'e}s.
\newblock Preface [{S}pecial issue: {N}umerical approximations of hyperbolic
  systems with source terms and applications].
\newblock {\em J. Sci. Comput.}, 48(1-3):1--2, 2011.

\bibitem[AMT04]{AMT04}
Christos Arvanitis, Charalambos Makridakis, and Athanasios~E. Tzavaras.
\newblock Stability and convergence of a class of finite element schemes for
  hyperbolic systems of conservation laws.
\newblock {\em SIAM J. Numer. Anal.}, 42(4):1357--1393, 2004.

\bibitem[BKM12]{BKM12}
E.~B{\"a}nsch, F.~Karakatsani, and Ch. Makridakis.
\newblock A posteriori error control for fully discrete crank--nicolson
  schemes.
\newblock {\em SIAM Journal on Numerical Analysis}, 50(6):2845--2872, 2012.

\bibitem[Chi14]{Chi14}
Elisabetta Chiodaroli.
\newblock A counterexample to well-posedness of entropy solutions to the
  compressible euler system.
\newblock {\em Journal of Hyperbolic Differential Equations}, 11(03):493--519,
  2014.

\bibitem[Daf79]{Daf79}
C.~M. Dafermos.
\newblock The second law of thermodynamics and stability.
\newblock {\em Arch. Rational Mech. Anal.}, 70(2):167--179, 1979.

\bibitem[Daf10]{Daf10}
Constantine~M. Dafermos.
\newblock {\em Hyperbolic conservation laws in continuum physics}, volume 325
  of {\em Grundlehren der Mathematischen Wissenschaften [Fundamental Principles
  of Mathematical Sciences]}.
\newblock Springer-Verlag, Berlin, third edition, 2010.

\bibitem[DiP79]{Dip79}
Ronald~J. DiPerna.
\newblock Uniqueness of solutions to hyperbolic conservation laws.
\newblock {\em Indiana Univ. Math. J.}, 28(1):137--188, 1979.

\bibitem[DLS10]{DS10}
Camillo De~Lellis and L{\'a}szl{\'o} Sz{\'e}kelyhidi, Jr.
\newblock On admissibility criteria for weak solutions of the {E}uler
  equations.
\newblock {\em Arch. Ration. Mech. Anal.}, 195(1):225--260, 2010.

\bibitem[DMO07]{DMO07}
Andreas Dedner, Charalambos Makridakis, and Mario Ohlberger.
\newblock Error control for a class of {R}unge-{K}utta discontinuous {G}alerkin
  methods for nonlinear conservation laws.
\newblock {\em SIAM J. Numer. Anal.}, 45(2):514--538, 2007.

\bibitem[DPE12]{DE12}
Daniele~Antonio Di~Pietro and Alexandre Ern.
\newblock {\em Mathematical aspects of discontinuous Galerkin methods}.
\newblock Mathematiques et applications. Springer, Heidelberg, New York,
  London, 2012.
\newblock La couv. porte en plus : SMAI.

\bibitem[EGGH98]{EGGH98}
R.~Eymard, T.~Gallou{\"e}t, M.~Ghilani, and R.~Herbin.
\newblock Error estimates for the approximate solutions of a nonlinear
  hyperbolic equation given by finite volume schemes.
\newblock {\em IMA J. Numer. Anal.}, 18(4):563--594, 1998.

\bibitem[EJNT86]{EJNT86}
W.~H. Enright, K.~R. Jackson, S.~P. N{\o}rsett, and P.~G. Thomsen.
\newblock Interpolants for {R}unge-{K}utta formulas.
\newblock {\em ACM Trans. Math. Software}, 12(3):193--218, 1986.

\bibitem[For88]{For88}
Bengt Fornberg.
\newblock Generation of finite difference formulas on arbitrarily spaced grids.
\newblock {\em Math. Comp.}, 51(184):699--706, 1988.

\bibitem[GM00]{GM00}
Laurent Gosse and Charalambos Makridakis.
\newblock Two a posteriori error estimates for one-dimensional scalar
  conservation laws.
\newblock {\em SIAM J. Numer. Anal.}, 38(3):964--988, 2000.

\bibitem[GMP15]{GMP_15}
Jan Giesselmann, Charalambos Makridakis, and Tristan Pryer.
\newblock A posteriori analysis of discontinuous galerkin schemes for systems
  of hyperbolic conservation laws.
\newblock {\em SIAM J. Numer. Anal.}, 53:1280--1303, 2015.

\bibitem[GR96]{GR96}
Edwige Godlewski and Pierre-Arnaud Raviart.
\newblock {\em Numerical approximation of hyperbolic systems of conservation
  laws}, volume 118 of {\em Applied Mathematical Sciences}.
\newblock Springer-Verlag, New York, 1996.

\bibitem[HH02]{HH02}
Ralf Hartmann and Paul Houston.
\newblock Adaptive discontinuous {G}alerkin finite element methods for
  nonlinear hyperbolic conservation laws.
\newblock {\em SIAM J. Sci. Comput.}, 24(3):979--1004 (electronic), 2002.

\bibitem[Hig91]{Hig91}
Desmond~J. Higham.
\newblock Runge-{K}utta defect control using {H}ermite-{B}irkhoff
  interpolation.
\newblock {\em SIAM J. Sci. Statist. Comput.}, 12(5):991--999, 1991.

\bibitem[HW08]{HW08}
Jan~S. Hesthaven and Tim Warburton.
\newblock {\em Nodal discontinuous {G}alerkin methods}, volume~54 of {\em Texts
  in Applied Mathematics}.
\newblock Springer, New York, 2008.
\newblock Algorithms, analysis, and applications.

\bibitem[JR05]{JR05}
Vladimir Jovanovi{\'c} and Christian Rohde.
\newblock Finite-volume schemes for {F}riedrichs systems in multiple space
  dimensions: a priori and a posteriori error estimates.
\newblock {\em Numer. Methods Partial Differential Equations}, 21(1):104--131,
  2005.

\bibitem[JR06]{JR06}
Vladimir Jovanovi{\'c} and Christian Rohde.
\newblock Error estimates for finite volume approximations of classical
  solutions for nonlinear systems of hyperbolic balance laws.
\newblock {\em SIAM J. Numer. Anal.}, 43(6):2423--2449 (electronic), 2006.

\bibitem[KLY10]{KLY10}
H.~Kim, M.~Laforest, and D.~Yoon.
\newblock An adaptive version of {G}limm's scheme.
\newblock {\em Acta Math. Sci. Ser. B Engl. Ed.}, 30(2):428--446, 2010.

\bibitem[KO00]{KO00}
Dietmar Kr{\"o}ner and Mario Ohlberger.
\newblock A posteriori error estimates for upwind finite volume schemes for
  nonlinear conservation laws in multidimensions.
\newblock {\em Math. Comp.}, 69(229):25--39, 2000.

\bibitem[Kr{\"o}97]{Kro97}
Dietmar Kr{\"o}ner.
\newblock {\em Numerical schemes for conservation laws}.
\newblock Wiley-Teubner Series Advances in Numerical Mathematics. John Wiley \&
  Sons Ltd., Chichester, 1997.

\bibitem[Kru70]{Kru70}
S.~N. Kru{\v{z}}kov.
\newblock First order quasilinear equations with several independent variables.
\newblock {\em Mat. Sb. (N.S.)}, 81 (123):228--255, 1970.

\bibitem[Laf04]{Laf04}
M.~Laforest.
\newblock A posteriori error estimate for front-tracking: systems of
  conservation laws.
\newblock {\em SIAM J. Math. Anal.}, 35(5):1347--1370, 2004.

\bibitem[Laf08]{Laf08}
M.~Laforest.
\newblock An a posteriori error estimate for {G}limm's scheme.
\newblock In {\em Hyperbolic problems: theory, numerics, applications}, pages
  643--651. Springer, Berlin, 2008.

\bibitem[Lap67]{Lapidus:1967:ArtVisc}
Arnold Lapidus.
\newblock A detached shock calculation by second-order finite differences.
\newblock {\em Journal of Computational Physics}, 2(2):154--177, November 1967.

\bibitem[LeV02]{LeV02}
Randall~J. LeVeque.
\newblock {\em Finite volume methods for hyperbolic problems}.
\newblock Cambridge Texts in Applied Mathematics. Cambridge University Press,
  Cambridge, 2002.

\bibitem[Mak07]{Mak07}
Charalambos Makridakis.
\newblock Space and time reconstructions in a posteriori analysis of evolution
  problems.
\newblock In {\em E{SAIM} {P}roceedings. {V}ol. 21 (2007) [{J}ourn\'ees
  d'{A}nalyse {F}onctionnelle et {N}um\'erique en l'honneur de {M}ichel
  {C}rouzeix]}, volume~21 of {\em ESAIM Proc.}, pages 31--44. EDP Sci., Les
  Ulis, 2007.

\bibitem[MN06]{MN06}
Charalambos Makridakis and Ricardo~H. Nochetto.
\newblock A posteriori error analysis for higher order dissipative methods for
  evolution problems.
\newblock {\em Numer. Math.}, 104(4):489--514, 2006.

\bibitem[PS11]{PS11}
Gabriella Puppo and Matteo Semplice.
\newblock Numerical entropy and adaptivity for finite volume schemes.
\newblock {\em Commun. Comput. Phys.}, 10(5):1132--1160, 2011.

\bibitem[QKS06]{Qiu:2006:DGFLUXES}
Jianxian Qiu, Boo~Cheong Khoo, and Chi-Wang Shu.
\newblock A numerical study for the performance of the runge-kutta
  discontinuous galerkin method based on different numerical fluxes.
\newblock {\em J. Comput. Phys.}, 212(2):540--565, March 2006.

\bibitem[RSS13]{Reisner:2013:ArtVisc}
J.~Reisner, J.~Serencsa, and S.~Shkoller.
\newblock A space-time smooth artificial viscosity method for nonlinear
  conservation laws.
\newblock {\em Journal of Computational Physics}, 235(??):912--933, February
  2013.

\bibitem[TB00]{Toro:2000:FLIC}
E.~F. Toro and S.~J. Billett.
\newblock Centred {TVD} schemes for hyperbolic conservation laws.
\newblock {\em IMA Journal of Numerical Analysis}, 20(1):47--79, January 2000.

\bibitem[ZS06]{ZS06}
Qiang Zhang and Chi-Wang Shu.
\newblock Error estimates to smooth solutions of {R}unge-{K}utta discontinuous
  {G}alerkin method for symmetrizable systems of conservation laws.
\newblock {\em SIAM J. Numer. Anal.}, 44(4):1703--1720 (electronic), 2006.

\end{thebibliography}

\end{document}